\documentclass[reqno,12pt]{article}
\newcommand{\NN}{\ensuremath{\mathbb{N}}}           
\newcommand{\PP}{\ensuremath{\mathbb{P}}}           
\renewcommand{\AA}{\ensuremath{\mathbb{A}}}           
\newcommand{\QQ}{\ensuremath{\mathbb{Q}}}           
\newcommand{\cA}{\ensuremath{\mathcal{A}}}

\newcommand{\cI}{\ensuremath{\mathcal{I}}}
\newcommand{\cJ}{\ensuremath{\mathcal{J}}}
\newcommand{\cL}{\ensuremath{\mathcal{L}}}
\newcommand{\cM}{\ensuremath{\mathcal{M}}}
\newcommand{\cO}{\ensuremath{\mathcal{O}}}
\newcommand{\cR}{\ensuremath{\mathcal{R}}}
\newcommand{\cS}{\ensuremath{\mathcal{S}}}
\newcommand{\SHom}{\ensuremath{\mathcal{H}om}}

\newcommand{\SL}{\operatorname{SL}}

\newcommand{\Proj}{\operatorname{Proj}}

\newcommand{\supp}{\operatorname{supp}}
\newcommand{\Hom}{\operatorname{Hom}}

\usepackage{amssymb,euscript,amsmath, amsthm, array}
\usepackage{mathtools}
\usepackage{makeidx}
\usepackage{enumerate}
\usepackage[margin=3cm]{geometry}
\usepackage[draft]{graphicx}
\usepackage{mathpazo}

\usepackage{verbatim}

\numberwithin{equation}{section}

\newtheorem{Theorem}{Theorem}[section]
\newtheorem{Definition}[Theorem]{Definition}
\newtheorem{Example}[Theorem]{Example}
\newtheorem{Lemma}[Theorem]{Lemma}
\newtheorem{Proposition}[Theorem]{Proposition}
\newtheorem{Corollary}[Theorem]{Corollary}
\newtheorem{Remark}[Theorem]{Remark}

\newcommand{\Bl}{\operatorname{Bl}}

\title{ Maximal compatible splitting and diagonals\\
  of Kempf varieties} \author{ Niels Lauritzen and Jesper Funch
  Thomsen }

\begin{document}
\maketitle

\begin{abstract} 
  Lakshmibai, Mehta and Parameswaran (LMP) introduced the notion of
  maximal multiplicity vanishing in Frobenius splitting.  In this
  paper we define the algebraic analogue of this concept and construct
  a Frobenius splitting vanishing with maximal multiplicity on the
  diagonal of the full flag variety. Our splitting induces a diagonal
  Frobenius splitting of maximal multiplicity for a special class of
  smooth Schubert varieties first considered by Kempf. Consequences
  are Frobenius splitting of tangent bundles, of blow-ups along the
  diagonal in flag varieties along with the LMP and Wahl conjectures
  in positive characteristic for the special linear group.
\end{abstract}

\section{Introduction}

In \cite{LakshmibaiMehtaParameswaran1998}, Lakshmibai, Mehta and
Parameswaran introduced the notion of multiplicities of Frobenius
splittings: if $X$ is a smooth projective algebraic variety over an
algebraically closed field $k$ of positive characteristic $p$, duality
for the Frobenius morphism identifies Frobenius splittings with
certain sections of the $(p-1)$-th power of the anticanonical line
bundle $\omega_X^{-1}$ on $X$.  If $Y\subseteq X$ is a compatibly
split smooth subvariety of codimension $d$ under the section $s$ of
$\omega_X^{1-p}$, then $s$ vanishes with multiplicity
$\leq (p-1) d$ on $Y$. The splitting $s$ is said to split $Y$
compatibly with maximal multiplicity if $s$ vanishes with
multiplicity $(p-1) d$ on $Y$ (cf.~\S \ref{fbups} of this paper for an
equivalent algebraic notion).  A Frobenius splitting vanishing
with maximal multiplicity on $Y$ lifts to a Frobenius
splitting of the blow-up $\Bl_Y(X)$ splitting the exceptional divisor
compatibly.

Let $X = G/P$, where $G$ is a semisimple linear algebraic group and
$P\subset G$ a parabolic subgroup.  In a beautiful geometric argument
Lakshmibai, Mehta and Parameswaran proved that a Frobenius splitting
of the blow-up $\Bl_\Delta(X\times X)$ compatibly splitting the
exceptional divisor implies Wahl's conjecture in positive
characteristic.  They conjectured the existence of a Frobenius
splitting of $X\times X$ vanishing  with maximal
multiplicity on the diagonal $\Delta$ (we refer to this as the LMP
conjecture, cf. \S2.4 in \cite{LakshmibaiMehtaParameswaran1998} and \S
2.C in \cite{BrionKumar2005}).

Wahl's conjecture predicts that the (generalized) Gaussian map
(cf.~\cite{Wahl1991})
\begin{equation}\label{Gauss}
H^0(X\times X, \cI_\Delta\otimes p_1^*\cL_1\otimes p_2^*\cL_2)\rightarrow
H^0(X, \Omega^1_X\otimes \cL_1\otimes \cL_2)
\end{equation}
is surjective for $\cL_1$
and $\cL_2$ ample line bundles on $X$. This conjecture was proved
by Kumar \cite{Kumar1992} for complex semisimple groups using detailed
information on the decomposition of tensor products. In positive
characteristic the conjecture has been proved for Grassmannians by
Mehta and Parameswaran \cite{MehtaParameswaran1997}, for symplectic
and orthogonal Grassmannians by Lakshmibai, Raghavan and Sankaran
\cite{lakshmibaietal2009} and by Brown and Lakshmibai for minuscule
$G/P$ \cite{BrownLakshmibai2008}. These positive characteristic
results were proved by verifying the LMP conjecture in the specific
cases. The LMP conjecture for $G/P$ is implied by the conjecture for
the full flag variety $G/B$ (cf.~Proposition \ref{Prop:pdm} of this paper).
Lakshmibai, Mehta and Parameswaran verified their conjecture
for $\SL_n/B$ and $n\leq 6$.

In this paper we prove the LMP conjecture for $\SL_n/P$ by explicitly
constructing a Frobenius splitting of $\SL_n/B\times \SL_n/B$
vanishing  with maximal multiplicity on the diagonal for
every $n\geq 2$. Our splitting compatibly splits $X\times X$, where
$X$ is a Kempf variety in $\SL_n/B$ (Kempf varieties are special smooth
Schubert varieties introduced by Kempf in \cite{Kempf1976}. See also
\S \ref{Kempf} in this paper for their definition and examples).

Our construction comes from observing in the $\SL_3$-case that the
product of the minors from the lower left hand corner in
$$
\begin{pmatrix}
x_{31} & 0 & x_{32} & 0 & x_{33} & 0\\
x_{21} & 0 & x_{22} & 0 & x_{23} & 0\\
x_{11} & 0 & x_{12} & 0 & x_{13} & 0\\
x_{11} & y_{11} & x_{12} & y_{12} & x_{13} & y_{13}\\
x_{21} & y_{21} & x_{22} & y_{22} & x_{23} & y_{23}\\
x_{31} & y_{31} & x_{32} & y_{32} & x_{33} & y_{33}
\end{pmatrix},
$$
where
$$
\begin{pmatrix}
x_{11} & x_{12} & x_{13}\\
x_{21} & x_{22} & x_{23}\\
x_{31} & x_{32} & x_{33}
\end{pmatrix},\,
\begin{pmatrix}
y_{11} & y_{12} & y_{13}\\
y_{21} & y_{22} & y_{23}\\
y_{31} & y_{32} & y_{33}
\end{pmatrix}\in
\SL_3
$$

\bigskip\noindent is a section of the anticanonical bundle on
$\SL_3/B\times \SL_3/B$ giving a Frobenius splitting vanishing
 with maximal multiplicity on the diagonal and compatibly
splitting $X\times X$, where $X$ is one of the five Kempf varieties in
$\SL_3/B$ (cf.~Example \ref{Idea} in this paper).

In the last part (\S \ref{WahlKempf}) of this paper, we enhance the
geometric arguments in \cite{LakshmibaiMehtaParameswaran1998} and show
that the Gaussian map \eqref{Gauss} is surjective, provided that
$\cL_1 = \cL\otimes \cM_1$ and $\cL_2=\cL\otimes \cM_2$, where $\cL$
is ample and $\cM_1, \cM_2$ globally generated line bundles on $X$ (a
projective smooth variety) and the diagonal $\Delta\subset X\times X$
is maximally compatibly split. Here we do not need the underlying
field to have odd characteristic (as in
\cite{LakshmibaiMehtaParameswaran1998}).  This enables us to prove
Wahl's conjecture also for Kempf varieties, since they posses unique
minimal ample line bundles as Schubert varieties in $G/B$. We do not
know, even over the complex numbers, if Wahl's conjecture holds for
smooth Schubert varieties.

We have found it very difficult to prove the LMP conjecture in a
general Lie theoretic context and hope this paper will add to the
inspiration for further research in this direction. We feel
nevertheless, that Frobenius splitting of tangent bundles (cf.  the
already known case of the cotangent bundle
\cite{KumarLauritzenThomsen1999}), diagonal Frobenius splitting of
Kempf varieties along with the LMP and Wahl conjecture for the special
linear group are of some interest.

We thank an anonymous referee for careful reading and pointing out
several sharpenings in our manuscript.

\newcommand{\diag}{\operatorname{diag}}

\section{Preliminaries}

A scheme will refer to a seperated scheme of finite type over an
algebraically closed field $k$ of characteristic $p>0$. A variety will
refer to a reduced scheme.

\subsection{The vanishing multiplicity on a smooth subvariety}
\label{mult-van}

Let $X$ be a smooth variety of dimension $n$, $\cL$ a line bundle on
$X$ and $Y\subset X$ a smooth subvariety of codimension $d$. Then the
blow-up $B=\Bl_Y(X)$ is a smooth variety and the exceptional divisor
$E\subset B$ a prime divisor. Let $s$ be a section of $\cL$. The
vanishing multiplicity of $s$ on $Y$ is defined as $v_E(\pi^* s)$ (in
the notation of \cite[II.6]{Hartshorne}), where $\pi:B\rightarrow X$
is the projection. Notice that the vanishing multiplicity of $s$ on
$Y$ can be computed locally on an open subset $U\subset X$ with $U\cap
Y\neq \emptyset$. Locally this definition is easy to handle: if $P\in
Y$, then there exists a regular system of parameters $x_1, \dots, x_n$
in $\cO_{X, P}$, such that $Y$ is defined by $I=(x_1, \dots,
x_d)$  \cite[VIII. Theorem 26]{ZariskiSamuel1975}. The vanishing multiplicity of $s$ is the maximal $m\geq 0$ with
$s_P\in I^m\cL$.

\subsection{Frobenius splitting}

We recall the crucial definitions and concepts on Frobenius splitting
from \cite{BrionKumar2005} with a few added generalizations on
Frobenius splitting of $\cO_X$-algebras along with the notion of
\textit{maximally compatibly split subschemes}. 

The \textit{absolute Frobenius morphism} on a scheme $X$ is the
morphism $F:X\rightarrow X$, which is the identity on point spaces and
the Frobenius homomorphism on the structure sheaf $\cO_X$. A
\textit{Frobenius splitting} of $X$ is an $\cO_X$-linear map
$\sigma:F_* \cO_X\rightarrow \cO_X$ splitting $F^\#:\cO_X\rightarrow
F_*\cO_X$. Another way of saying this, is that $\sigma$ is a group
homomorphism $\cO_X\rightarrow \cO_X$ satisfying
\begin{itemize}
\item $\sigma(f^p g) = f \sigma(g)$
\item $\sigma(1) = 1$
\end{itemize}
locally on open subsets.  A Frobenius split scheme has to be reduced.
A closed subscheme $Y\subset X$ is called \textit{compatibly split}
under a Frobenius splitting $\sigma$ if
$$
\sigma(F_* \cI_Y)\subset \cI_Y.
$$
The following very useful results follow (almost) from first
principles (cf.~\cite[Proposition 1.2.1 and Lemma
1.1.7]{BrionKumar2005}).
\begin{Proposition}\label{Prop:F}
Let $\sigma$ be a Frobenius splitting of a scheme $X$ and
let $Y$ and $Z$ be compatibly split subschemes of $X$ under $\sigma$.
\begin{enumerate}[(i)]
\item\label{Prop:Fi}
The irreducible components of $Y$ are compatibly split under $\sigma$.
\item\label{Prop:Fii}
The scheme theoretic intersection $Y\cap Z$ given by $\cI_Z+\cI_Y$ is 
compatibly split under $\sigma$.
\item
The scheme theoretic union $Y\cup Z$ given by $\cI_Z\cap\cI_Y$ is 
compatibly split under $\sigma$.
\item\label{Prop:Fiv}
  If $U$ is a dense open subscheme of a reduced scheme $X$, and if
  $$\sigma\in \Hom_{\cO_X}(F_*\cO_X, \cO_X)$$ restricts to a splitting
  of $U$, then $\sigma$ is a splitting of $X$. If, in addition, $Y$ is
  a reduced closed subscheme of $X$ such that $U\cap Y$ is dense in
  $Y$ and compatibly split by $\sigma|_U$, then $Y$ is compatibly
  split by $\sigma$.
\end{enumerate}
\end{Proposition}

\subsection{Frobenius splitting of $\cO_X$-algebras}\label{fbups}

The Frobenius homomorphism makes perfect sense for a sheaf $\cA$ of
$\cO_X$-algebras, where $X$ is a scheme. In analogy with the classical
definition we define $\cA$ to be Frobenius split if there exists a
homomorphism
$$
\sigma: F_* \cA\rightarrow \cA
$$
of $\cA$-modules splitting the Frobenius homomorphism $\cA\rightarrow
\cA$. Similarly we call a sheaf of ideals $\cJ$ in $\cA$ compatibly 
split under $\sigma$ if $\sigma(F_* \cJ)\subset \cJ$.

We let 
\begin{align*}
\cR(\cI) &= \bigoplus_{m\geq 0} \cI^m t^m = \cO_X[\cI t]\\
&=\{a_0+a_1 t + \cdots + a_n t^n \mid a_j \in \cI^j\}\subset \cO_X[t]
\end{align*} 
denote the \textit{Rees algebra} corresponding to a sheaf of ideals
$\cI\subset \cO_X$. The sheaf of ideals $\cI \cR(\cI)$ is called
the \textit{exceptional ideal}.

A Frobenius splitting $\sigma:F_*\cO_X\rightarrow \cO_X$ can always be
extended to the Frobenius splitting
$
\sigma[t]:F_* \cO_X[t]\rightarrow \cO_X[t]
$
given by 
$$
\sigma[t](a_0 + a_1 t + \cdots) := \sigma(a_0) + \sigma(a_p) t +
\sigma(a_{2p}) t^2 + \cdots
$$
\begin{Definition}
  Let $\sigma:F_* \cO_X\rightarrow \cO_X$ be a Frobenius splitting of
  $X$. A closed subscheme $Y\subset X$ is called maximally compatibly split
  under $\sigma$ if
$$
\sigma(\cI^{n p +1})\subset \cI^{n+1}
$$
for every $n\geq 0$, where $\cI$ is the ideal sheaf defining $Y$.
\end{Definition}

Notice that a maximally compatibly split scheme is compatibly split
and that $\sigma(\cI^{np}) \subset \cI^n$ for $n\geq 0$.
The following result can be checked explicitly by reducing to the
affine case.

\begin{Proposition}\label{Propmaxsplitalg}
  Let $Y\subset X$ be a maximally compatibly split closed subscheme
  under a Frobenius splitting $\sigma:F_* \cO_X\rightarrow \cO_X$ and
  $\cI$ the ideal sheaf defining $Y$.
\begin{enumerate}[(i)]

\item
  Then
  $\sigma[t]$ restricts to a Frobenius splitting of the Rees algebra
  $\cR(\cI)$ compatibly splitting the exceptional ideal $\cI
  \cR(\cI)$.
\item If furthermore $Z$ is a compatibly split closed subscheme under
  $\sigma$, then the induced splitting on $Z$ splits $Y\cap Z$
  maximally, where $Y\cap Z$ denotes the scheme theoretic intersection.
\end{enumerate}
\end{Proposition}

The \textit{blow-up} of a scheme $X$ along a closed subscheme $Y$
given by the ideal sheaf $\cI$ is defined as $\Bl_Y(X) := \Proj
\cR(\cI)$. The exceptional ideal identifies with the inverse image
ideal sheaf $\pi^{-1}(\cI)$, under the canonical morphism
$\pi:\Bl_Y(X)\rightarrow X$. It is an invertible sheaf defining the
\textit{exceptional divisor} of $\pi$. In this setting we will prove
the following analogue of Proposition \ref{Propmaxsplitalg}.

\begin{Proposition}\label{xyz}
  Let $Y\subset X$ be a maximally compatibly split closed subscheme
  under a Frobenius splitting $\sigma:F_* \cO_X\rightarrow \cO_X$.
\begin{enumerate}[(i)]

\item 
Then $\sigma$ extends to a Frobenius splitting of the blow-up
  $\Bl_Y(X)$ compatibly splitting the exceptional divisor.
\item If the closed subscheme $Z$ is compatibly split under $\sigma$,
  then the induced splitting on $Z$ extends to a Frobenius splitting
  of $\Bl_{Y\cap Z}(Z)$ splitting the exceptional divisor compatibly,
  where $Y\cap Z$ denotes the scheme theoretic intersection.
\end{enumerate}
\end{Proposition}

Proposition \ref{xyz} is a consequence of the next subsection, where we give the
necessary details for turning a Frobenius splitting of a homogeneous
$\cO_X$-algebra $\cA$ into a Frobenius splitting of the scheme $\Proj
\cA$.

\subsection{Graded Frobenius splittings}

For a commutative ring $R$ of characteristic $p > 0$ and an $R$-module
$M$ with scalar multiplication $(r,m)\mapsto rm$, we let $F_* M$
denote the $R$-module coinciding with $M$ as an abelian group but with
scalar multiplication $(r, m)\mapsto r^p m$.

Let $ S = S_0\oplus S_1 \oplus \cdots $ be a graded noetherian ring of
characteristic $p$, such that $F_* S$ is a finitely generated
$S$-module. If $M = M_0 \oplus M_1 \oplus \cdots$ is a graded
$S$-module, then we have a direct sum decomposition of $F_* M$ into
graded $S$-modules
$$
F_* M = F_*M^{(0)} \oplus \cdots \oplus F_*M^{(p-1)},
$$
where
$$
F_*M^{(j)}=\bigoplus_{i\equiv j\,\, (\text{mod}\, p)} M_i,
$$
for $j = 0, \dots, p-1$. An element $m\in M_{n p +
  j}\subset F_* M^{(j)}$ has degree $n$.

\begin{Lemma}\label{Lemma:tildeF}
Let $X=\Proj(S)$ and $F:X\rightarrow X$ be the absolute Frobenius
morphism on $X$. Then there is a canonical isomorphism
$$
\widetilde{F_* M^{(0)}} \cong F_* \widetilde{M}.
$$
\end{Lemma}
\begin{proof}
Let $f\in S$ be a homogeneous element. Then 
$$
\varphi_f\left(\frac{m}{f^n}\right) = \frac{m}{f^{n p}}
$$ defines a local isomorphism $(F_*
M^{(0)})_{(f)} \rightarrow F_* (M_{(f)})$ on $D_+(f)$. The isomorphisms
$\varphi_f$ patch up to give the desired global isomorphism.
\end{proof}

\begin{Example}
  Suppose that $S = k[x_0, x_1,\dots, x_n]$, where $k$ is a field of
  characteristic $p$. Then there is an isomorphism
$$
F_* S^{(0)} \cong S\oplus S(-1)^{\ell_1} \oplus \cdots \oplus
S(-n)^{\ell_n},
$$ of graded $S$-modules for certain $\ell_1, \dots, \ell_n\in \NN$. Lemma
\ref{Lemma:tildeF}
shows that 
$$ F_* \cO_X \cong \cO_X\oplus
\cO_X(-1)^{\ell_1}\oplus\cdots\oplus\cO_X(-n)^{\ell_n}
$$ for $X = \PP^n_k = \Proj(S)$. In particular, it follows that
$\PP^n_k$ is Frobenius split. Building a monomial basis for
$F_*S^{(0)}$ in degrees $0, p, 2p, \dots, n p$ we also have the following
recursive formula for $\ell_j$:
$$
\ell_j = \binom{j p + n}{n} - \sum_{i=1}^j \binom{i + n}{n}\ell_{j-i},
$$
where $j = 0, \dots, n$.
The fact that $F_*\cO_{\PP^n}(m)$ splits into a direct sum of line
bundles is a classical result due to Hartshorne (cf.~\cite[\S
6]{Hartshorne1970}).
\end{Example}

\subsubsection{Frobenius splitting of $\Proj(S)$}

For $\sigma\in \Hom_S(F_*S, S)$, we let
$$
\sigma_0\in \Hom_S(F_* S^{(0)}, S)_0
\subset \Hom_S(F_* S^{(0)}, S)
$$ 
denote the degree $0$ component of $\sigma$ restricted to $F_*
S^{(0)}$. Then $\sigma_0: F_* S^{(0)}\rightarrow S$ is a homomorphism
of graded $S$-modules.  We may view $\sigma_0\in \Hom_S(F_* S, S)$
satisfying $\sigma_0(S_{n p})\subset S_n$ and $\sigma_0(S_m) = 0$ if
$p\nmid m$.

\begin{Lemma}\label{Lemma:homogsplit}
  Suppose $\sigma\in \Hom_S(F_* S, S)$, where $S = S_0\oplus S_1\oplus
  \cdots$ is a graded ring. Then $\sigma_0$ is a Frobenius splitting if
  $\sigma$ is a Frobenius splitting.  If $I\subset S$ is
  a homogeneous ideal, then $\sigma_0$ splits $I$ compatibly if 
  $\sigma$ splits $I$ compatibly.

  If $S$ is Frobenius split, then $X = \Proj(S)$ is Frobenius split.
  If $I$ is a compatibly split homogeneous ideal, then the closed
  subscheme $Y=\Proj(S/I)$ is compatibly split in $X$.
\end{Lemma}
\begin{proof}
  Let $\sigma: F_* S\rightarrow S$ be a Frobenius splitting. Clearly
  $\sigma(1)_0 = \sigma_0(1)$, so that $\sigma_0$ is a Frobenius
  splitting if $\sigma$ is.  Notice that
  $\sigma(I)\subset I$ implies $\sigma_0(I)\subset I$, since
  $\sigma_0(x) = \sigma(x)_n$ for $x\in S_{n p}$.  Now the statements
  in the first part of the lemma follow.  For the second part let
  $\cI\subset \cO_X$ be the ideal sheaf defining $Y$. Then $ \cI =
  \widetilde{I}$ and $\cO_X = \widetilde{S}$. Now Lemma
  \ref{Lemma:tildeF} gives
\begin{align*}
F_* \cI &= \widetilde{F_* I^{(0)}},\\
F_* \cO_X &= \widetilde{F_* S^{(0)}}.
\end{align*}  
The graded $S$-homomorphism $\sigma_0:F_* S^{(0)}\rightarrow S$ then
  gives a Frobenius splitting $\widetilde{\sigma_0}:F_* \cO_X
  \rightarrow \cO_X$ with $\widetilde{\sigma_0}(F_* \cI)\subset \cI$.
\end{proof}

\begin{Corollary}\label{Corollary:homogsplit}
  Let $X$ be a scheme, $\cS = \cS_0\oplus \cS_1\oplus \cdots$ a sheaf
  of graded $\cO_X$-algebras and $\cI\subset \cS$ a homogeneous ideal.
  We will assume that $F_* \cS$ locally is a finitely generated
  $\cS$-module. If $\cS$ is Frobenius split compatibly with $\cI$,
  then $\Proj(\cS)$ is Frobenius split compatibly with
  $\Proj(\cS/\cI)$.
\end{Corollary}
\begin{proof}
  Let $\sigma: F_* \cS\rightarrow\cS$ be a Frobenius splitting of
  $\cS$ with $\sigma(F_* \cI)\subset \cI$.  The construction of
  $\sigma_0$ globalizes to give a Frobenius splitting
  $\sigma_0:F_*\cS^{(0)}\rightarrow \cS$ (with $\sigma_0(\cS_{n
    p})\subset \cS_n$ and $\sigma_0(\cS_m) = 0$ for $p\nmid m$). For
  an affine open subset $U\subset X$, $\sigma_0$ gives by Lemma
  \ref{Lemma:homogsplit} a Frobenius splitting
$$
\sigma_U:F_* \cO_{\Proj(\cS(U))}\rightarrow \cO_{\Proj(\cS(U))}
$$
compatibly splitting the closed subscheme $\Proj(\cS(U)/\cI(U))$.
Coming from the global splitting $\sigma_0$, these splittings patch up
to give the desired global splitting of $\Proj(\cS)$.
\end{proof}

\subsection{Duality for the Frobenius morphism}

On a non-singular variety $X$
duality for the Frobenius morphism $F:X\rightarrow X$ is available for
the study of Frobenius splitting: there is a functorial isomorphism
$F_* \omega_X^{1-p}\rightarrow \SHom_{\cO_X}(F_*\cO_X, \cO_X)$, where
$\omega_X$ is the canonical line bundle on $X$. In
\cite{MehtaRamanathan1985}, it is shown how geometric properties of
the zero divisor of a section of $\omega_X^{1-p}$ translate into
properties of compatible Frobenius splitting. To recall this powerful
result in more precise terms, we need to introduce some notation.

If $\alpha\in \QQ\setminus \NN$ and $x$ is a variable, we define
$x^\alpha := 0$. 
Now let $x = (x_1,\dots, x_n)$ denote a regular
system of parameters (in a regular local ring) and $\alpha = (\alpha_1, \dots,
\alpha_n)\in\QQ^n$ a rational vector.  Then we define
$$
x^\alpha := x_1^{\alpha_1}\cdots x_n^{\alpha_n}.
$$
and $x^\gamma:= x_1^\gamma \cdots x_n^\gamma$ for $\gamma\in \QQ$. 
\begin{Theorem}[Mehta and Ramanathan \cite{MehtaRamanathan1985}]\label{Theorem:MR}
  Let $X$ be a non-singular variety of dimension $n$ over an
  algebraically closed field $k$ of characteristic $p$.  Then
  there is a canonical isomorphism
$$
\partial: F_* \omega_X^{1-p}\rightarrow \SHom_{\cO_X}(F_*\cO_X, \cO_X)
$$
of $\cO_X$-modules whose completion 
$$
\hat{\partial}_P: F_*\omega_{\hat{R}}^{1-p}\rightarrow
\Hom_{\hat{R}}(F_* \hat{R}, \hat{R})
$$
at a closed point $P\in X$, is given by
$$
\hat{\partial}_P\left(x^\alpha \frac{1}{(d x)^{p-1}}\right)(x^\beta) =
x^{(\alpha + \beta + 1)/p -1}, 
$$
where $\hat{R}=k[[x_1,\dots, x_n]]$ and $x_1,\dots, x_n$ is a regular
system of parameters in $R := \cO_{X, P}$ with $d x = dx_1\wedge \cdots \wedge dx_n$.
\end{Theorem}

\begin{Remark}\label{RemarkCompl}
  Notice that $\partial(s)$ in Theorem \ref{Theorem:MR} is a Frobenius
  splitting if and only if $\partial(s)(1) = 1$. This translates into a
  local condition on the section $s$. Suppose
$$
s_P=(\sum_\alpha a_\alpha x^\alpha) (1/ dx)^{p-1}
$$
is a local expansion of $s$ at $P\in X$. Let $\supp(s_P)$ denote the
exponents of the monomials occurring with non-zero coefficient in
$s_P$.  For $\partial(s_P)$ to be a Frobenius splitting we must have
$p-1\in \supp(s_P)$ and $p-1 + p v\not\in \supp(s_P)$ for $v\in
\NN^n\setminus\{0\}$. If $X$ is complete, then $\partial(s)$ is a
Frobenius splitting if and only if $p-1\in \supp(s_P)$ for some $P\in
X$ \cite[Proposition 6]{MehtaRamanathan1985} .
\end{Remark}

An important consequence of this result is the following
\cite[Proposition 1.3.11]{BrionKumar2005}.

\begin{Lemma}\label{Lemma:Z}
  Let $X$ be a complete smooth variety.  If $\sigma$ is a section of
  $\omega_X^{-1}$ such that $\partial(\sigma^{p-1})$ is a Frobenius
  splitting of $X$, then the subscheme of zeros, $Z(\sigma) \subset
  X$, is compatibly split under $\partial(\sigma^{p-1})$.
\end{Lemma}

We have the following result analogous to 
\cite[Proposition 2.1]{LakshmibaiMehtaParameswaran1998}. In the proof
we use the notation
$$
|\alpha| := \alpha_1 +\cdots + \alpha_n
$$
for a vector $\alpha = (\alpha_1, \dots, \alpha_n)\in \QQ^n$.

\begin{Lemma}\label{Lemma:LMP}
  Let $Z$ be a non-singular variety of dimension $n$ and $W\subset Z$
  a non-singular subvariety of codimension $d$. Let $s$ be a section
  of $\omega_Z^{1-p}$, such that $\partial(s)$ is a Frobenius
  splitting of $Z$. Then $s$ vanishes with multiplicity $\leq d(p-1)$
  on $W$. The section $s$ vanishes with maximal multiplicity $d(p-1)$
  on $W$ if and only if $W$ is maximally compatibly split under
  $\partial(s)$.
\end{Lemma}

\begin{proof}
  Let $z_1, \dots, z_n$ be a regular system of parameters in
  $R:=\cO_{Z, P}$, where $P\in W$. We may assume that the ideal
  $I\subset R$ defining $W$ at $P$ is given by $x:=(z_1,\dots, z_d)$.
  Define $y:=(z_{d+1}, \dots, z_n)$ and let
\begin{equation}\label{masp}
t = (\sum a_{\alpha,\beta}x^\alpha y^\beta)\left(\frac{1}{d x\wedge d
    y}\right)^{p-1}
\end{equation}
be the local expansion of $s$ at $P$ in the completion $k[[z_1,\dots,
z_n]]$ of $R$. If $s$ vanishes with multiplicity $> d(p-1)$ on $W$, then the
term $x^{p-1} y^{p-1}$ cannot occur with non-zero coefficient in \eqref{masp}
contradicting that $\partial(s)$ is a Frobenius splitting. Therefore
$s$ vanishes with multiplicity $\leq d(p-1)$ on $W$.

Assume that $t$ vanishes with multiplicity $d(p-1)$ on $W$.  This
means that $|\alpha|\geq d(p-1)$ for every $\alpha$ with $a_{\alpha,
  \beta}\neq 0$ in \eqref{masp}. We will prove that $\partial(t)(I^{m p+1})
\subset I^{m+1}$ for $m\geq 0$. For this we assume that
$$
w=\sum c_{\gamma, \delta} x^\gamma y^\delta\in I^{m p+1}
$$
i.e. $|\gamma|\geq m p+1$ for every $\gamma$ with $c_{\gamma,
  \delta}\neq 0$. Now we have
$$
|(\alpha + \gamma + 1)/p - 1| \geq \frac{d(p-1) + m p + 1 + d}{p} - d = m + \frac{1}{p}.
$$
So if the vector $(\alpha + \gamma + 1)/p$ is integral, then $|(\alpha
+ \gamma + 1)/p - 1|\geq m+1$. This shows that $\partial(t)(w)\in
I^{m+1}$ recalling the definition of $\partial(t)$ in Theorem
\ref{Theorem:MR}.

Now assume that $\partial(t)(I^{m p+1}) \subset I^{m+1}$ for $m\geq
0$. We will prove that $t$ has to vanish with multiplicity $d(p-1)$ on
$W$. Suppose that $|\alpha| < d (p-1)$ for some non-zero $a_{\alpha, \beta}$ in \eqref{masp}. Let $m_i\in \NN$ be given by
$$
m_i (p-1) \leq \alpha_i < (m_i + 1)(p-1)
$$
for $i = 1, \dots, d$ and similarly $m_j(p-1) \leq \beta_j < (m_j+1)(p-1)$ for 
$j=d+1, \dots, n$.
Define the monomial $x^\gamma y^\delta\in I$ by
$$
\gamma = ((m_1 + 1) p - \alpha_1 -1, \dots, (m_d+1)p-\alpha_d-1)
$$
and similarly $\delta = ((m_{d+1} + 1)p - \beta_{d+1} - 1, \dots, 
(m_n + 1)p - \beta_n-1)$. Then
$$
\partial(x^\alpha y^\beta\, (\frac{1}{d x\wedge d
    y})^{p-1})(x^\gamma y^\delta)\in I^{m_1 + \cdots + m_d}\setminus
I^D,
$$
where $D = m_1 + \cdots + m_d + 1$.
But $x^\gamma y^\delta\in I^{(D-1)p + 1}$, since
\begin{align*}
\sum_{i=1}^d ((m_i+1) p - \alpha_i - 1) &>
\sum_{i=1}^d (m_i+1)p - d (p-1) - d = \sum_{i=1}^d m_i p\\
&=(D-1)p.
\end{align*}
This contradicts our assumption and we must have $|\alpha|\geq d
(p-1)$ for every non-zero $a_{\alpha, \beta}$ in \eqref{masp}.
\end{proof}

The following remark relates to the issue of Frobenius splitting of
the tangent bundle on a Frobenius split variety (cf.~our remarks in
the end of the introduction).

\begin{Remark}
  If $X$ is smooth and $X\times X$ is Frobenius split with the
  diagonal $\Delta_X\subset X\times X$ maximally compatibly split,
  then the tangent bundle $T_X$ on $X$ is Frobenius split, since the
  exceptional divisor in $\Bl_{\Delta_X}(X\times X)$ is isomorphic to
  $ \PP(T_X)$ \cite[Lemma 1.1.11]{BrionKumar2005}. 
\end{Remark}

We also need the following (\cite[Proposition
2.3]{LakshmibaiMehtaParameswaran1998} and \cite[Exercises
1.3.E.(13)]{BrionKumar2005}).

\begin{Proposition}\label{Prop:pdm}
  Let $f:X\rightarrow Y$ be a proper morphism of smooth varieties with
  $f_*\cO_X = \cO_Y$. Let $Z\subset X$ be a smooth subvariety such
  that $f$ is smooth at some point of $Z$. If $X$ is Frobenius split
  and $Z$ compatibly split with maximal multiplicity, then the induced
  splitting of $Y$ has maximal multiplicity along the non-singular
  locus of $f(Z)$.
\end{Proposition}

\subsection{Residual normal crossing}

In this section we recall a very important concept introduced by
Mehta, Lakshmibai and Parameswaran \cite[Definition
1.6]{LakshmibaiMehtaParameswaran1998}.

\begin{Definition}
  A power series $f \in k [[x_1, \ldots ., x_n]]$ is said to have residual
  normal crossings if either
  \begin{itemize}
    \item $n = 0$ and $f \neq 0$ or
    
    \item $n > 0,$ $x_1 | f$ and $f / x_1 + (x_1) \in k [[x_1, \ldots ,
      x_n]] / (x_1) \simeq k [[x_2, \ldots ., x_n]]$ has residual
      normal crossing in $k[[x_2, \dots, x_n]]$.
  \end{itemize}
\end{Definition}

The definition of residual normal crossings is dependent on the
ordering of the variables i.e.~when stating that $f\in k[[x_1, \dots,
x_n]]$ has residual normal crossing, it is implicitly assumed that the
variables $x_1, \dots, x_n$ are ordered.

\begin{Example}
  The polynomial $f = x (z y - x^2) (w - y)\in k[[x,y,z,w]]$ has
  residual normal crossing. However, if the variables are 
  ordered $x, w, z, y$, then $f$ does not have residual normal crossing
  i.e.~$f\in k[[x,w,z,y]]$ does not have residual normal crossing.
\end{Example}

The minimal term in a residual normal crossing power series $f\in
k[[x_1, \dots, x_n]]$ is precisely $x_1 \cdots x_n$, when the
monomials are ordered according to the lexicographical ordering $<$
given by $x_n < x_{n-1} < \cdots < x_1$.
This implies the following result by Remark
\ref{RemarkCompl}.

\begin{Proposition}\label{Prop:rnc}
  Let $X$ be a complete smooth variety, $P\in X$ and $x_1, \dots, x_n$
  a system of parameters of $\cO_{X, P}$. If $s\in \Gamma(X,
  \omega_X^{-1})$, such that $s_P\in \hat{\cO}_{X, P} = k[[x_1,
  \dots, x_n]]$ has residual normal crossing, then $\partial(s^{p-1})$
  is a Frobenius splitting of $X$.
\end{Proposition}

\section{Group theory}

Let $G$ be a semisimple algebraic group, $B$ a Borel subgroup of $G$
and $P\supset B$ a parabolic subgroup. A \textit{Schubert variety} is
defined as the closure of a $B$-orbit in the generalized flag variety
$G/P$. The singular locus of a Schubert variety is $B$-stable.

The map $\pi:G\rightarrow G/B$ is
a locally trivial principal $B$-fibration and $G\times^B E\rightarrow
G/B$ is a vector bundle of rank $\dim_k E$, where $E$ is a finite
dimensional representation of $B$.  We let $\Gamma(E)$ denote the
global sections of this vector bundle i.e.
$$
\Gamma(E) = \{f: G\rightarrow E \mid f(x b) = b^{-1} f(x), \text{ for every }x\in G, b\in B\}.
$$
For a one dimensional representation $\chi$ of $B$ we get the following explicit description of the global sections of the line bundle $G\times^B\chi$ on $G/B$:
\begin{equation}
\label{Gamma}
\Gamma(\chi) = \{f\in k[G] \mid f(g b) = \chi(b)^{-1} f(g), \text{ for every }g\in G, b\in B\}.
\end{equation}

For the rest of this paper we will assume that $G=\SL_n(k)$ with $B$
equal to the upper triangular matrices containing the diagonal
matrices
$$
T = \{\diag(t_1, \dots, t_n) \mid t_i\in k,\, t_1 \cdots t_n = 1\},
$$
$U$ the unipotent upper triangular matrices and $U^-$ the unipotent
lower triangular matrices. The canonical map
$$
\pi:U^-\rightarrow \pi(U^-)\subset G/B
$$
identifies $U^-$ with an (affine) open subset of $G/B$, since $U^-\cap B =
\{e\}$.  This open subset is isomorphic to affine $n(n-1)/2$ -- space.

Furthermore, $B = T U$ and $X(B) = X(T)$, where $X(B)$ denotes the
one-dimensional representations (characters) of $B$.  Let
$\epsilon_i(t) = t_i$ for $t\in T$ and $\omega_i = \epsilon_1 + \cdots
+ \epsilon_i$ for $1\leq i \leq n-1$ be characters of $T$. Then $X(T)$
is a free abelian group of rank $n-1$ with basis $\omega_1, \dots,
\omega_{n-1}$. The canonical line bundle on $G/B$ can be identified
with $G\times^B (2\omega_1 + \cdots + 2\omega_{n-1})$.

In the next section we give an example showing the explicit nature of residual
normal crossings in constructing 
Frobenius splittings for $G/B$ vanishing with different multiplicities
on $B/B$.

\subsection{Frobenius splitting of $G/B$ by residual normal crossings}

For an $n\times n$-matrix $g\in G$ we let $ \delta_i(g)$ denote the
$i\times i$ minor from the lower left hand corner i.e. the minor
corresponding to the columns $\{1, \dots, i\}$ and rows $\{n, \dots,
n-i+1\}$ of $g$. Similarly we let $\delta'_i(g)$ denote the
(principal) $i\times i$ minor from the upper left hand corner i.e. the
minor corresponding to the columns $\{1, \dots, i\}$ and rows $\{1,
\dots, i\}$.

We let
$
\delta(g) = \delta_1(g) \cdots \delta_{n-1}(g)
$
and similarly
$
\delta'(g) = \delta'_1(g) \cdots \delta'_{n-1}(g).
$
Then $\delta, \delta'\in \Gamma(-\omega_1 - \cdots -
\omega_{n-1})$ and
\begin{equation}\label{cs}
s(g) = \delta(g) \delta'(g)
\end{equation}
is a section of the anticanonical line bundle.

\begin{Example}
  As a global section of the anticanonical line bundle \eqref{cs}
  identifies by \eqref{Gamma} with the regular function
  \begin{equation}\label{rf}
    f = x_{31} (x_{21} x_{32} - x_{31} x_{22})\,\, x_{11} (x_{11} x_{22} - x_{21} x_{12})\in k[\SL_3]
\end{equation}
for
$$
g = 
\begin{pmatrix}
x_{11} & x_{12} & x_{13}\\
x_{21} & x_{22} & x_{23}\\
x_{31} & x_{32} & x_{33}
\end{pmatrix}\in \SL_3
$$
and restricts to the function 
\begin{equation}\label{rfum}
x_{31} (x_{21} x_{32} - x_{31})
\end{equation}
on
$U^-$. This polynomial has residual normal crossing with respect to
$x_{31}, x_{21}, x_{32}$  proving that $f^{p-1}$
defines a Frobenius splitting of $\SL_3/B$ by Proposition \ref{Prop:F}\eqref{Prop:Fiv} and Proposition \ref{Prop:rnc}.

However, \eqref{rfum} does not vanish with maximal multiplicity on the
point $B/B$, since $(x_{21} x_{32} - x_{31})\not\in (x_{21}, x_{32},
x_{31})^2$. There is, however, a section with this maximal vanishing property:
$$
s = x_{21} x_{31} (x_{11} x_{32} - x_{31} x_{12})(x_{11} x_{22} - x_{21}
x_{12}),
$$
Specializing, it follows that $s$ restricts to the (residual) normal
crossing polynomial
$$
x_{21} x_{31} x_{32}
$$
on $U^-$. This idea can be generalized from $\SL_3$ to $\SL_n$ for $n>
3$. See \cite{LakshmibaiMehtaParameswaran1998} for this and a standard
monomial approach to constructing Frobenius splittings of maximal
multiplicity.
\end{Example}

\subsection{Kempf varieties}\label{Kempf}

In \cite{Kempf1976}, Kempf inspired many subsequent developments in
algebraic groups proving his celebrated vanishing theorem first for
the general linear group. Kempf considered a very natural class of
(smooth) Schubert varieties as stepping stones in an inductive proof.
Here we review the definition of these Schubert varieties from
\cite{Kempf1976}.

We let 
$$
A = \{(a_1, \dots, a_n)\in \NN^n \mid a_1 \geq a_2 \geq \cdots \geq a_n = 0,\,
n - a_j \geq j\ \text{for}\ j=1, \dots, n\}.
$$
For $a\in A$ we let $M(a)$ denote the closed subset
of $G$ given by
$$
\Bigg\{
\begin{pmatrix} 
x_{11} & \cdots & x_{1n}\\
\vdots & \ddots & \vdots\\
x_{n1} & \cdots & x_{nn}
\end{pmatrix}\in G
\Bigg|\, x_{ij} = 0\ \text{for}\ i > n - a_j\Bigg\}.
$$
This subset is $B\times B$-stable, as it is stable with respect to row
operations adding a multiple of a higher index row to a lower index
one and similarly adding a multiple of a lower index column to a
higher index one.  The Schubert variety $K(a)=\pi(M(a))\subset G/B$ is
called a \textit{Kempf variety} (see also
\cite{Lakshmibai1976}). Notice that $K((0, \dots, 0)) = G/B$ and
$K((n-1, n-2, \dots, 1, 0)) = B/B$. The codimension of $K(a)$ is $a_1+
\cdots a_n$. In particular the unique codimension one Kempf variety is
given by the vanishing of the lower left hand corner i.e.~$a = (1, 0,
\dots, 0)$. Kempf varieties are smooth as $U^-\cap K(a)$ is a linear
subspace of $U^-\cong \AA^{n(n-1)/2}$ and $U^-\cap K(a)$ is an open
subset of $K(a)$ containing $B/B$.

\begin{Example}
  The Kempf varieties corresponding to $(1, 0, 0, 0), (2, 1, 0, 0)$
  and $(2, 1, 1, 0)$ in $G = \SL_4$ are depicted below.
$$
\begin{pmatrix}
* & * & * & *\\
* & * & * & *\\
* & * & * & *\\
0 & * & * & *
\end{pmatrix}
\qquad
\begin{pmatrix}
* & * & * & *\\
* & * & * & *\\
0 & * & * & *\\
0 & 0 & * & *
\end{pmatrix}
\qquad
\begin{pmatrix}
* & * & * & *\\
* & * & * & *\\
0 & * & * & *\\
0 & 0 & 0 & *
\end{pmatrix}.
$$
Informally a placement of a lower triangular zero implies zeros below and to the left of
the zero.
\end{Example}

\subsubsection{Rectangular Kempf varieties}\label{RK}

Every Kempf variety arises as the scheme-theoretic intersection of
distinguished Kempf varieties, which we call \textit{rectangular Kempf
  varieties}. A rectangular Kempf variety $K(r)$ of height $t\leq n-1$
is given by
$$
r\in \{a\in A\setminus\{0\}\mid a_i\in \{0, t\}\ \text{for}\ i = 1, \dots, n\}.
$$
The \textit{width} of a Kempf variety $K(r)$ is the number of non-zero
entries in $r$.

\begin{Example}
  The rectangular Kempf varieties of heights one and two for $\SL_4$
  are depicted below: \medskip
$$
\begin{pmatrix}
* & * & * & *\\
* & * & * & *\\
* & * & * & *\\
0 & * & * & *
\end{pmatrix},\quad
\begin{pmatrix}
* & * & * & *\\
* & * & * & *\\
* & * & * & *\\
0 & 0 & * & *
\end{pmatrix},\quad
\begin{pmatrix}
* & * & * & *\\
* & * & * & *\\
* & * & * & *\\
0 & 0 & 0 & *
\end{pmatrix},\quad
\begin{pmatrix}
* & * & * & *\\
* & * & * & *\\
0 & * & * & *\\
0 & * & * & *
\end{pmatrix},\quad
\begin{pmatrix}
* & * & * & *\\
* & * & * & *\\
0 & 0 & * & *\\
0 & 0 & * & *
\end{pmatrix}
$$
\medskip 

They correspond to the defining vectors $(1,0,0,0),\,(1,1,0,0),\,
(1,1,1,0),\, (2,0,0,0), (2,2,0,0)$ and widths $1, 2, 3, 1, 2$
respectively.
\end{Example}

\begin{Lemma}\label{Lemma:RK}
  For $G = \SL_n$ there are $n(n-1)/2$ rectangular Kempf
  varieties. Every Kempf variety is the scheme-theoretic intersection of rectangular
  Kempf varieties.
\end{Lemma}

\section{Matrix calculations}

In this section we outline the rather explicit linear algebra which is
the basis of our diagonal Frobenius splitting of $\SL_n/B\times
\SL_n/B$. 

We let $\delta_i(M)$ denote the $i\times i$ minor
from the lower left hand corner in a matrix $M$.
For two $n\times n$ matrices 
$$
g =
\begin{pmatrix}
x_{11} & \cdots & x_{1n}\\
\vdots & \ddots & \vdots\\
x_{n1} & \cdots & x_{nn}
\end{pmatrix}\qquad\text{and}
\qquad h = 
\begin{pmatrix}
y_{11} & \cdots & y_{1n}\\
\vdots & \ddots & \vdots\\
y_{n1} & \cdots & y_{nn}
\end{pmatrix}
$$
in $G$ we define
the 
$2n\times 2n$ matrix
$$
M(g,h) = 
\begin{pmatrix}
x_{n1} &  0 & x_{n2} & 0 & \cdots & x_{nn} & 0\\
\vdots &  \vdots & \vdots & \vdots & \ddots & \vdots & \vdots\\
x_{21} & 0 & x_{22} & 0 & \cdots & x_{2n} & 0\\
x_{11} & 0 & x_{12} & 0 & \cdots & x_{1n} & 0\\
x_{11} &  y_{11} & x_{12} & y_{12} & \cdots & x_{1n} & y_{1n}\\
x_{21} &  y_{21} & x_{22} & y_{22} & \cdots & x_{2n} & y_{2n}\\
\vdots &  \vdots & \vdots & \vdots & \ddots & \vdots & \vdots\\

x_{n1} &  y_{n1} & x_{n2} & y_{n2} & \cdots & x_{nn} & y_{nn}
\end{pmatrix}
$$
with determinant $\pm 1$. Notice that $\delta_i(M(g, h))$
is invariant under right
translation by $U\times U$ for $1\leq i \leq 2n$.
We are interested in 
the lower $n\times 2n$ submatrix 
$$
L(g, h)=\begin{pmatrix}
1 & 1 & 0 & 0 & \cdots & 0 & 0\\
x_{21} & y_{21} & 1 & 1 & \cdots & 0 & 0\\
\vdots & \vdots & \vdots & \vdots & \cdots & \vdots & \vdots\\
x_{n-1, 1} & y_{n-1, 1} & x_{n-1, 2} & y_{n-1, 2} & \cdots & 0 & 0\\
x_{n1} &  y_{n1} & x_{n2} & y_{n2} & \cdots & 1 & 1
\end{pmatrix}
$$
of $M(g, h)$ for $g, h\in U^-$.

\begin{Definition}\label{DefOdd} The following definitions are necessary to introduce our Frobenius splitting.
\begin{enumerate}[(i)]
\item
For $1 \leq i \leq n$ we define $L_i(g, h) = \delta_i(L(g,h))$.
\item
When $n \leq i \leq 2n+1$ we define $L_i(g,h)$ to be the 
$(2n-i) \times (2n-i)$-submatrix of $L(g,h)$ obtained by 
deleting the first $2(i-n)$ columns and the first $(i-n)$ 
rows from the first $i$ columns of $L(g,h)$. 
\item
For $1\leq i \leq 2n-1$, we let
$V_i$ denote the variables in the diagonal of $L_i$, 
$M_i$ the monomial ideal
generated by them and
$m_i$ the
monomial given by their product. For $i=0$, we define $V_0=\emptyset$ and $M_0 = (0)$.
\end{enumerate}
\end{Definition}

The reader is advised to study the following example illustrating
these definitions.

\begin{Example}
For $G = \SL_4$ and $g, h\in U^-$,
$$
L(g, h)=\begin{pmatrix}
1 & 1 & 0 & 0 & 0 & 0 & 0 & 0\\
x_{21} & y_{21} & 1 & 1& 0 & 0 & 0 & 0\\
x_{31} & y_{31} & x_{32} & y_{32} & 1 & 1 & 0 & 0\\
x_{41} & y_{41} & x_{42} & y_{42} & x_{43} & y_{43} & 1 & 1
\end{pmatrix}.
$$
Here
$$
L_1 = 
\begin{pmatrix}
x_{41}
\end{pmatrix},\quad
L_2 = 
\begin{pmatrix}
x_{31} & y_{31}\\
x_{41} & y_{41}
\end{pmatrix},
\quad
L_3 = 
\begin{pmatrix}
x_{21} & y_{21} &  1\\
x_{31} & y_{31} & x_{32}\\
x_{41} & y_{41} & x_{42}
\end{pmatrix},
\quad
L_4 =
\begin{pmatrix}
1 & 1 & 0 & 0\\
x_{21} & y_{21} &  1 & 1\\
x_{31} & y_{31} & x_{32} & y_{32}\\
x_{41} & y_{41} & x_{42} & y_{42}
\end{pmatrix}
$$
and
$$
L_5 = 
\begin{pmatrix}
 1 & 1 & 0\\
x_{32} & y_{32} & 1\\
x_{42} & y_{42} & x_{43}
\end{pmatrix},
\quad
L_6 = 
\begin{pmatrix}
 1 & 1\\
 x_{43} & y_{43}
\end{pmatrix},
\quad
L_7 =
\begin{pmatrix}
1      \\

\end{pmatrix}.
$$
\medskip

Notice that 
\begin{align*}
V_1 &= \{x_{41}\}\\
V_2 &= \{x_{31}, y_{41}\}\\
V_3 &= \{x_{21}, y_{31}, x_{42}\}\\
V_4 &= \{y_{21}, x_{32}, y_{42}\}\\
V_5 &= \{y_{32}, x_{43}\}\\
V_6 &= \{y_{43}\}\\
V_7 &= \emptyset
\end{align*}
and that 
$$
\det L_i \equiv m_i \mod M_1 + \cdots + M_{i-1}
$$
for $i = 1, \dots, 7$. Notice also that the columns in $L_i$ are 
pairwise identical in the set of  variables $\{ x_{ij} \}$ and $\{ y_{ij} \}$. This 
ensures that the determinants of the $L_i$'s will vanish with 
high multiplicity on the diagonal in $U^-\times U^-$.

\end{Example}

To prepare for showing that $\delta(M(g,h))^{p-1}$ is a Frobenius
splitting section of the anticanonical bundle on $G/B\times G/B$
we need the following result when restricting to the open affine
subset $U^-\times U^-$.

\begin{Proposition}\label{Prop:fs}
For $1 \leq i \leq 2n-1$ and $g, h\in U^-$, we have
\begin{enumerate}[(i)]

\item\label{itemid}
\begin{equation}\label{congrM}
\delta_i(M(g, h)) = \det L_i(g, h) 
\end{equation}
\item\label{itemdiagmult}
$$
\det L_i(g, h)\in I_\Delta^{\mu(i)},
$$
where $I_\Delta \subseteq k[U^-\times U^-]$ is the ideal defining the
diagonal,
$$
\mu(i)=\min\Bigg(\Bigg\lfloor\frac{i}{2}\Bigg\rfloor, \, 
\Bigg\lfloor\frac{2 n - i}{2}\Bigg\rfloor\Bigg)
$$
and $\lfloor x \rfloor$ denotes the largest integer $\leq x$
\item\label{itemvars}
$$
V_1 \cup \cdots \cup V_{2n-1} = \{x_{n1}, x_{n-1, 1}, y_{n1}, \dots,
x_{n, n-1}, y_{n, n-1}\}.
$$
\item\label{itemmon}
$$
\det L_i(g, h) \equiv m_i \mod M_1 + \cdots + M_{i-1}.
$$

\end{enumerate}
\end{Proposition}
\begin{proof}
For $1\leq i \leq n$, \eqref{congrM}
  is clear.  When $i>n$ and $g, h \in U^-$ the $i \times i$-submatrix
of $M(g,h)$ in the lower left hand corner will have a lower triangular 
unipotent structure in the top $2(i-n)$ rows (up to row permutation of these
rows). In particular, when computing  the determinant $\delta_i(M(g, h))$ 
one might as well start by deleting the first $2(i-n)$ columns and rows. 
The connection with ${\rm det}(L_i(g,h))$ is then clear.

  The proof of \eqref{itemdiagmult} follows from pairwise subtraction
  of columns, before computing the determinant, using the fact that
  $\mu(i)$ is  the number
  of identical $x$-columns and $y$-columns in $L_i(g, h)$.

Let 
$$
\Delta_r(g, h) = 
\begin{cases}
\{L(g, h)_{ij} \mid i-j = n-r\} &\text{for}\quad r = 1, \dots, n\\
 \{L(g, h)_{ij} \mid j-i = n-r\} &\text{for}\quad r = n+1, \dots, 2n-1
\end{cases}
$$
denote the $2n-1$ ``diagonals'' in $L(g, h)$
  starting with the lower left hand corner. Then \eqref{itemvars}
follows from the fact that
  $V_i$ picks up the variables in $\Delta_i(g, h)$ for $i = 1, \dots, 2n - 1$. 

  In evaluating the determinant of $L_i(g, h)$, a term different from
  the product of the diagonal elements always involves a variable in
  $V_1\cup \cdots \cup V_{i-1}$ for $i = 1, \dots, 2n-1$. This implies
  \eqref{itemmon}.
\end{proof}

\section{The diagonal Frobenius splitting on  $\SL_n/B\times \SL_n/B$}

The following simple lemma is the fundamental tool for showing compatible
splitting for Kempf varieties.

\begin{Lemma}\label{Lemma:monideal}
Let $f, g\in k[x_{m+1}, \dots, x_n]$ be relatively prime polynomials. Then
$$
(x_1, \dots, x_m, f g) = (x_1, \dots, x_m, f)\cap (x_1, \dots, x_m, g)
$$
in $k[x_1, \dots, x_n]$.
\end{Lemma}

To get an initial grasp of our diagonal Frobenius splitting, the reader
is encouraged to look at the following example.

\begin{Example}\label{Idea}
For $G = \SL_3$ and $g, h\in U^-$, $f:=\delta(M(g, h))$ is
$$
f = 
\delta\Bigg(
\begin{pmatrix}
x_{31} & 0 & x_{32} & 0 & 1 & 0\\
x_{21} & 0 & 1 & 0 & 0 & 0\\
1      & 0  & 0 & 0 & 0 & 0\\
1 & 1 & 0 & 0 & 0 & 0\\
x_{21} & y_{21} & 1 & 1 & 0 & 0\\
x_{31} & y_{31} & x_{32} & y_{32} & 1 & 1
\end{pmatrix}
\Bigg)
$$
Here 
\begin{align*}
f &= \det L_1(g, h) \det L_2(g,h) \det L_3(g, h) \det L_4(g, h)\\
&= x_{31} (x_{21}y_{31} - x_{31}y_{21})
(y_{21} x_{32} - y_{31} - x_{21} x_{32} + x_{31}) (y_{32} - x_{32})
\end{align*}
and $f\in k[x_{31}, x_{21}, y_{31}, y_{21}, x_{32}, y_{32}]$ has
residual normal crossing. Furthermore $f$ vanishes with multiplicity
three on the diagonal $V(y_{31} - x_{31}, y_{32} - x_{32}, y_{21} -
x_{21})$ as $\mu(1) + \mu(2) + \mu(3) + \mu(4) = 0 + 1 + 1 + 1 = 3$
(cf.~Proposition \ref{Prop:fs}\eqref{itemdiagmult}).  Therefore
$f^{p-1}$ is a Frobenius splitting of $\SL_3/B\times \SL_3/B$ by
Proposition \ref{Prop:rnc}
vanishing with maximal multiplicity on the diagonal. Lemma
\ref{Lemma:Z} and Proposition \ref{Prop:F}\eqref{Prop:Fi} show that
the ideals
$$
(x_{31}),\quad (x_{21}y_{31} - x_{31}y_{21}),\quad
(y_{21} x_{32} - y_{31} - x_{21} x_{32} + x_{31})\quad\text{and}\quad
(y_{32} - x_{32})
$$
are compatibly split by $f^{p-1}$. Consequently
$$
(x_{31}, x_{21} y_{31})  = (x_{31}) + (x_{21}y_{31}-x_{31}y_{21})
$$
is compatibly split by Proposition
\ref{Prop:F}\eqref{Prop:Fii} and 
$$
(x_{31}, x_{21}),\quad (x_{31}, y_{31})
$$
are compatibly split by Lemma \ref{Lemma:monideal}. Similarly
\begin{align*}
&(x_{31}, y_{31}, x_{21}) + (y_{21} x_{32} - y_{31} - x_{21} x_{32} + x_{31})=\\
&(x_{31}, y_{31}, x_{21}, y_{21} x_{32})=\\
&(x_{31}, y_{31}, x_{21}, y_{21})\cap (x_{31}, y_{31}, x_{21}, x_{32})
\end{align*}
showing that $(x_{31}, y_{31}, x_{21}, y_{21})$ is compatibly
split. Along the same lines we get that
\begin{align*}
&(x_{31}, y_{31}) + (y_{21} x_{32} - y_{31} - x_{21} x_{32} + x_{31})=\\
&(x_{31}, y_{31}, x_{32}(y_{21}-x_{21})) = \\
&(x_{31}, y_{31}, x_{32})\cap (x_{31}, y_{31}, y_{21}-x_{21})
\end{align*}
and $(x_{31}, y_{31}, x_{32})$ is compatibly split showing that
$$
(x_{31}, y_{31}, x_{32})+(y_{32}-x_{32}) = (x_{31}, y_{31}, x_{32}, y_{32})
$$
is compatibly split. 

We have verified that $X\times X\subset \SL_3/B\times \SL_3/B$ is
compatibly split, where $X$ is any rectangular Kempf variety.
\end{Example}

With this example in mind, we  state and prove our main result.

\begin{Theorem}
\label{thm}
For $g, h\in U^-\subset SL_n$, let
$$
f = \delta(M(g, h))\in k[U^-\times U^-]\cong k[V_1\cup \cdots \cup V_{2n-1}]. 
$$ 
Then 
\begin{enumerate}[(i)]

\item\label{rncd1} $f$ is a residual normal crossing polynomial when
  the variables are ordered respecting $V_1, V_2, \dots, V_{2n-1}$:
  if $x\in V_i$ and $y\in V_j$ are variables and $i < j$, then $x$ must precede $y$ in the ordering of the variables.
\item\label{mult1}
$f$ vanishes with multiplicity $\geq n(n-1)/2$ on the diagonal $\Delta_{U^-}$.
\item\label{Fs} Let $\omega$ denote the canonical line bundle on
  $\SL_n/B\times\SL_n/B$. Then
$$
\delta(M(g,h))^{p-1}\in k[G\times G]
$$
is a Frobenius splitting section of $\omega^{1-p}$ vanishing with
maximal multiplicity on $\Delta_{G/B}$ compatibly splitting $X\times
X$, where $X$ is a Kempf variety.
\end{enumerate}
\end{Theorem}
\begin{proof}
Proposition \ref{Prop:fs}\eqref{itemmon} shows \eqref{rncd1}.
Since
$$
\sum_{i=1}^{2n-1}\, \mu(i) = \frac{n(n-1)}{2},
$$
Proposition \ref{Prop:fs}\eqref{itemid} and Proposition \ref{Prop:fs}\eqref{itemdiagmult}
imply $(ii)$.

Let us prove \eqref{Fs}. The regular function $\delta(M(g, h))\in
k[G\times G]$ is invariant under right translation by $U\times
U$. This amounts to observing that the column operations on $g$ and
$h$ coming from right multiplication by $U\times U$ do not change
$\delta_i(M(g, h))$ for $1\leq i \leq 2n-1$.

Define $\omega_0 = \omega_n = 0$. Then $\delta_{2i}(M(g, h))\in
\Gamma(-\omega_i, -\omega_i)$ and $\delta_{2i-1}(M(g,h))\in \Gamma(-\omega_i,
-\omega_{i-1})$ for $1\leq i \leq n$.  This shows that $\delta(M(g,
h))\in k[G\times G]$ is a section of the anticanonical line bundle
$G\times^B(2\omega_1 + \cdots + 2 \omega_n)$ on $G/B\times G/B$.  Now
\eqref{rncd1} and \eqref{mult1} show after restricting to $U^-\times
U^-$ that $\delta(M(g, h))^{p-1}$ is a Frobenius splitting vanishing
with (maximal) multiplicity $(p-1) n(n-1)/2$ on $\Delta_{G/B}$.  We
have silently applied Proposition \ref{Prop:F}\eqref{Prop:Fiv}, 
Proposition \ref{Prop:rnc} and the fact that vanishing multiplicity  
can be checked on an open subset (cf. Section \ref{mult-van})

It remains to show that $X\times X$ is compatibly split, where
$X\subset \SL_n/B$ is a Kempf variety. We can assume by Lemma
\ref{Lemma:RK} that $X$ is a rectangular Kempf variety (the argument works for
general Kempf varieties, but is slightly less clear).

Suppose that $X$ is of height $r$ and width $s$. Then  we must show that the monomial ideal
generated by the variables
$$
V_X=\Big\{x_{ij}, y_{ij}\, \Big|\, n-r < i \leq n,\, 1\leq j\leq s\Big\}
$$
is compatibly split under $f^{p-1}$. We will prove that the monomial
ideal generated by the variables
$$
V_X \cap (V_1 \cup \cdots \cup V_m)
$$
is compatibly split by induction on $m$. Since $V_X\cap
V_1 = \{x_{n1}\}$ and $x^{p-1}_{n1}$ is the first factor in $f^{p-1}$,
compatible splitting holds for $m=1$. Suppose now that the monomial ideal
generated by
$$
W:=V_X\cap (V_1 \cup \cdots \cup V_m)\subsetneq V_X
$$
is compatibly split. Then $(W, \delta_{m+1}(M(g, h))) = (W, D)$, where
$D$ is a monomial of the form $d m $, where $m$ is the product
of the variables $V_X\cap V_{m+1}$ and $d$ is a
monomial. This is a consequence of the formula
$$
\det
\begin{pmatrix}
A & B\\
0 & C
\end{pmatrix} =
\det(A) \det(C),
$$
where $A, B$ and $C$ are compatible block matrices.

It follows by
Lemma \ref{Lemma:monideal} that the ideal generated by
$$
V_X \cap (V_1 \cup \cdots \cup V_m\cup V_{m+1})
$$
is compatibly split. Since $V_X\subset V_1\cup \cdots \cup V_N$ for
$N\geq r+s-1$ the result follows.
\end{proof}

\section{Wahl's conjecture for Kempf varieties}\label{WahlKempf}

Let $Z$ denote a smooth projective variety. The sheaf of 
differentials on $Z$ is defined by 
$$ \Omega^1_Z = \mathcal I_{\Delta} / \mathcal I_{\Delta}^2,$$
where $\mathcal I_\Delta$ denotes the sheaf of ideals 
defining the diagonal within $Z \times Z$. In this setup
we may consider the   quotient 
morphism
$$ \mathcal I_\Delta \rightarrow  
\mathcal I_{\Delta} / \mathcal I_{\Delta}^2 = 
\Omega^1_Z.$$
Fixing line bundles $\mathcal L_1$ and $\mathcal L_2$
on $Z$ we obtain an induced  restriction 
morphism  
\begin{equation} 
\label{gauss}
{\rm H}^0\big(Z \times Z, \mathcal I_\Delta \otimes
(\mathcal L_1 \boxtimes \mathcal L_2) \big)
\rightarrow {\rm H}^0\big(Z ,\Omega^1_Z 
\otimes  
\mathcal L_1 \otimes \mathcal L_2 \big),
\end{equation} 
where $\cL_1\boxtimes \cL_2:=p_1^*\cL_1\otimes p_2^*\cL_2$ and $p_1, p_2: X\times X\rightarrow X$ are the projections on the
first and second factors.  In case $Z$ is a flag variety and $\mathcal
L_1$ and $\mathcal L_2$ are ample it has been conjectured by J. Wahl
\cite{Wahl1991} that the map (\ref{gauss}) is surjective. In
characteristic zero this is now a theorem proved by S. Kumar
\cite{Kumar1992}.  In positive characteristic only sporadic cases are
known as outlined in the introduction.

The aim of the last part of this paper is to obtain the following
related and seemingly stronger result

\begin{Theorem}
\label{conjecture}
Assume that the blow-up ${\rm Bl}_{\Delta}(Z \times Z)$ admits 
a Frobenius splitting which is compatible  with $E_Z$. 
Let $\mathcal L$ denote a very ample line bundle on $Z$ 
and let $\mathcal M_1$ and $\mathcal M_2$ denote globally 
generated line bundles on $Z$. Let $j>0$ denote an integer.
Then the natural map from
$${\rm H}^0\big(Z \times Z, \mathcal I_\Delta^{j} \otimes
((\mathcal L^j\otimes \mathcal M_1) \boxtimes (\mathcal L^j \otimes 
\mathcal M_2 ))\big),$$
to 
$$ {\rm H}^0\big(Z  , S^j \Omega^1_Z \otimes
\mathcal L^{2j} \otimes \mathcal M_1 
\otimes \mathcal M_2 \big),$$
induced by the identification 
$\mathcal I_\Delta^{j}/ \mathcal I_\Delta^{j+1}
 = S^j \Omega^1_Z$, is surjective.
\end{Theorem}

Notice that when $Z$ admits a minimal ample line bundle $\mathcal L$; 
i.e. an ample line bundle on $Z$ such that every line bundle of the 
form $\mathcal M \otimes \mathcal L^{-1}$, with $\mathcal M$ ample, 
is globally generated, then Wahl's conjecture is a consequence of 
Theorem \ref{conjecture}. Schubert varieties are examples of 
varieties admitting minimal ample line bundle. When the Schubert 
variety is a flag variety this is well known; e.g.  in the notation of the
previous sections the minimal ample line bundle on $G/B$ is defined
by the weight 
$$-\rho = - (\omega_1 + \cdots \omega_{n-1}).$$
For a general Schubert variety the claim follows by the fact that any ample line
bundle on a Schubert variety may be lifted to an ample line bundle
on the flag variety containing the Schubert variety \cite[Prop.2.2.8]{Brion} 

With these remarks in place the following corollary now follows from 
Proposition \ref{xyz} and Theorem \ref{thm}

\begin{Corollary}
The conjecture of Wahl on the surjectivity of the  
map (\ref{gauss}) is satisfied  for Kempf varieties  $Z$ 
and  ample line bundles $\mathcal L_1$ and $\mathcal L_2$.
\end{Corollary}

The rest of this paper is concerned with the proof of Theorem 
\ref{conjecture}. The proof is highly inspired by the discussion 
in Section 3 of \cite{LakshmibaiMehtaParameswaran1998}. 
As a side result we obtain certain cohomological vanishing 
results for smooth varieties admitting  various types of 
Frobenius splitting (cf. Prop. \ref{prop1} and Prop. \ref{prop2}); 
e.g. for Kempf varieties.  We start by collecting a number of
well known results about blow-ups along diagonals.

\subsection{Blow-up of ${{\mathbb P}^N \times {\mathbb P}^N}$ along
the diagonal}

Consider the variety ${\mathbb P^N}={\mathbb P}(V)$ with 
homogeneous coordinates 
$X_0, \dots, X_N$. The homogeneous ideal defining the diagonal within
the product ${{\mathbb P}^N \times {\mathbb P}^N}$ is generated
by the elements
$$ X_{i,j} = X_i \otimes X_j - X_j \otimes X_i~~, 0 \leq i < j \leq N;$$
all of the same multidegree $(1,1)$.  Applying the Rees algebra 
description of the blow-up this leads to an embedding 
of ${\rm Bl}_\Delta(
{\mathbb P}^N \times {\mathbb P}^N)$ as a closed subvariety of 
the product
\begin{equation}
\label{product}
{{\mathbb P}^N \times {\mathbb P}^N} \times {\mathbb P}^{{N+1 \choose 2}-1}.
\end{equation}
Alternatively one could also obtain this embedding by considering 
 ${\rm Bl}_\Delta({\mathbb P}^N \times {\mathbb P}^N)$
as the graph of the rational morphism
\begin{equation}
\label{rational}
{{\mathbb P}^N \times {\mathbb P}^N} \dashrightarrow 
{\mathbb P}^{{N+1 \choose 2}-1},
\end{equation}
defined by the generators $X_{i,j}$ of the diagonal ideal (cf.~\cite[Ex. 7.18]{Harris}). The latter description makes it 
evident that  ${\rm Bl}_\Delta({\mathbb P}^N \times {\mathbb P}^N)$
is contained within 
\begin{equation}
\label{product2}
{\mathbb P}^N \times {\mathbb P}^N  \times {\rm Gr}_2(V),
\end{equation}
where ${\rm Gr}_2(V)$ denotes the Grassmannian of planes 
in $V$ with the Pl\"ucker embedding  in  ${\mathbb P}^{ {N+1 \choose 2} -1} $.
This also explains the following setwise description 
of the blow-up 
\begin{equation}
\label{set-descrip}
 {\rm Bl}_\Delta({\mathbb P}^N \times {\mathbb P}^N)
= \{ (l_1, l_2 ,b) \in  {\mathbb P}^N \times {\mathbb P}^N  \times {\rm Gr}_2(V)
~ : ~  l_1,l_2 \subset b \}.
\end{equation}
In this setting the exceptional divisor $E$ is determined as the 
set of points 
$$ E =  \{ (l ,b) \in  {\mathbb P}^N  \times {\rm Gr}_2(V)
~ : ~  l \subset b \}
\subset  {\mathbb P}^N  \times {\rm Gr}_2(V),
$$
where we consider $ {\mathbb P}^N$ as being diagonally embedded 
in $ {\mathbb P}^N \times {\mathbb P}^N $.

The projection on the first two coordinates 
$$ \pi : {\rm Bl}_\Delta({\mathbb P}^N \times {\mathbb P}^N) \rightarrow
{{\mathbb P}^N \times {\mathbb P}^N},$$
is the  blow-up map. Restricting $\pi$ to the exceptional
divisor $E$ defines
the map
$$\pi_E : E \rightarrow {\mathbb P}^N,$$
coinciding with the projectivized tangent bundle on $\mathbb{P}^N$.
Finally we let 
$$ \tau :  {\rm Bl}_\Delta({\mathbb P}^N \times {\mathbb P}^N)
\rightarrow {\rm Gr}_2(V),$$
denote the map induced by projection on the third 
coordinate, while $\tau_E$ denotes its restriction to 
$E$. 

\begin{Lemma}
Let $\mathcal O_{2,V}(1)$ (resp. $\mathcal O(1)$) denote the 
ample generator of the  Picard group of ${\rm Gr}_2(V)$ (resp.
$\mathbb P^N$). Then as locally free sheaves
\begin{equation}
\label{iso-line} 
\tau^* ( \mathcal O_{2,V}(1)) \simeq 
\mathcal O(-E) \otimes \pi^* \big( \mathcal O(1) \boxtimes \mathcal O(1) \big).
\end{equation} 
\end{Lemma}
\begin{proof}
This follows from a local calculation but can also be obtained in
the following more abstract  way : assume, first of all, that $N\geq2$
in which case we have the following identity of Picard groups 
$$ {\rm Pic}\big({\rm Bl}_\Delta({\mathbb P}^N \times 
{\mathbb P}^N) \big) 
\simeq {\rm Pic}(  {\mathbb P}^N \times {\mathbb P}^N)
\oplus \mathbb Z.$$
In particular, we may find unique integers $c_1,c_2$ and $c_3$ 
such that 
$$  \tau^* ( \mathcal O_{2,V}(1))  \simeq 
\mathcal O(- c_1E) \otimes \pi^* \big( \mathcal O(c_2) \boxtimes \mathcal O(c_3) \big).$$
Restricting to the open subset ${\rm Bl}_\Delta({\mathbb P}^N \times 
{\mathbb P}^N) \setminus E \simeq  ({\mathbb P}^N \times 
{\mathbb P}^N) \setminus \Delta$ we determine $(c_2,c_3)$
as the bidegree of the rational morphism (\ref{rational}). In particular,
we find that $c_2=c_3=1$. To find $c_1$ we fix some line 
$\mathbb P^1$ inside $\mathbb P^N$ and consider 
$\mathbb P^1 \times \mathbb P^1$ as a closed subset 
of $ {\rm Bl}_\Delta({\mathbb P}^N \times 
{\mathbb P}^N) $ by identifying it with its strict 
transform.
As the rational morphism (\ref{rational}) is constant on an open  
dense subset of $\mathbb P^1 \times \mathbb P^1$, the same is true for
the restriction of $\tau$ to $\mathbb P^1 \times \mathbb P^1$. In particular, the restriction of the 
sheaf $\tau^* ( \mathcal O_{2,V}(1))$ to $\mathbb P^1 \times 
\mathbb P^1$ is trivial. Now as the sheaf of ideals of the diagonal in  
$\mathbb P^1 \times \mathbb P^1$ equals $\mathcal O(-1)
\boxtimes \mathcal O(-1)$ we conclude 
$$ -c_1 + c_2 = -c_1 + c_3 = 0.$$
Thus $c_1=1$. This ends the proof in case  $N \geq 2$. For $N=1$
the map $\tau$ is constant and the claimed isomorphism (\ref{iso-line})
is trivial.
\end{proof}

We claim that $\tau$ is a ${\mathbb P}^1 \times {\mathbb P}^1$-bundle. 
More precisely, let $b_0 \in {\rm Gr}_2(V)$ denote any plane in $V$
and let $P_0$ denote the stabilizer of $b_0$ in the group ${\rm SL}(V)$.
Then ${\rm Gr}_2(V)$ is isomorphic to the quotient ${{\rm SL}(V)}/{P_0}$
while $ {\rm Bl}_\Delta({\mathbb P}^N \times {\mathbb P}^N)$ may be described
as 
\begin{equation}
\label{equiv-bl}
{\rm Bl}_\Delta({\mathbb P}^N \times {\mathbb P}^N) = 
{\rm SL}(V) \times_{P_0} ({\mathbb P}(b_0) \times {\mathbb P}(b_0) ),
\end{equation}
where $P_0$ acts by the diagonal action on  ${\mathbb P}(b_0) \times {\mathbb P}(b_0) $.
Thus $\tau$ is just the natural map
$$ \tau : {\rm SL}(V) \times_{P_0} ({\mathbb P}(b_0) \times {\mathbb P}(b_0) )
\rightarrow {{\rm SL}(V)}/{P_0}.$$
In this notation we may describe the exceptional divisor
as 
$$ E = {\rm SL}(V) \times_{P_0} {\mathbb P}(b_0),$$  
where we think of $E$  as a subset of
(\ref{equiv-bl}) by embedding ${\mathbb P}(b_0)$ diagonally
in the product ${\mathbb P}(b_0) \times {\mathbb P}(b_0) $. It follows 
that the restriction
$$\tau_E : {\rm SL}(V) \times_{P_0} {\mathbb P}(b_0)
\rightarrow {{\rm SL}(V)}/{P_0},$$
is a ${\mathbb P}^1$-bundle over
${\rm Gr}_2(V)$. 

\subsection{Blow-up of diagonals in general}

Returning to the general case of a smooth projective subvariety 
$Z$ in ${\mathbb P}(V)$ we may consider the blow-up 
${\rm Bl}_\Delta (Z \times Z)$  as the strict transform of 
$Z \times Z$ in ${\rm Bl}_\Delta({\mathbb P}^N  \times 
{\mathbb P}^N)$. In particular, we obtain a closed 
embedding
\begin{equation}
\label{set-descritZ}
{\rm Bl}_\Delta(Z \times Z) \subset Z \times Z \times {\rm Gr}_2(V).
\end{equation}
The exceptional divisor $E_Z$ is thus embedded as
\begin{equation}
\label{set-descrip-EZ}
 E_Z \subset Z \times {\rm Gr}_2(V).
\end{equation}
In this setting the  blow-up morphism
$$ \pi_Z : {\rm Bl}_\Delta(Z \times Z) \rightarrow Z \times Z,$$
coincides with the projection on the first two coordinates, while 
its restriction 
$$\pi_{E_Z} : E_Z \rightarrow Z$$
coincides with the projectivized tangent bundle on $Z$.
Thus if we consider ${\rm Gr}_2(V)$ as the set of lines in 
${\mathbb P}^N= {\mathbb P}(V)$, then $E_Z$ consists 
of the set of pairs $(l,b) \in Z \times {\rm Gr}_2(V)$
such  that $b$ is a line tangent to the point $l$ in $Z$.

The projection on the third coordinate is denoted by 
$$\tau_Z : {\rm Bl}_\Delta(Z \times Z) \rightarrow {\rm Gr}_2(V),$$
while its restriction to $E_Z$ is denoted by $\tau_{E_Z}$.

\subsection{Fibres of $\tau_{E_Z}$}

By the discussion above the  fibre of $\tau_{E_Z}$ over a 
line $b$ in ${\mathbb P}(V)$ consists of the set of points 
$l$ in $Z$ such that $b$ is tangent to $Z$ at $l$. Thus
the following result is now easy to prove

 \begin{Lemma}
\label{linear subspace}
If every nonempty fibre of $\tau_{E_Z}$ has dimension $1$ 
then $Z$ coincides with ${\mathbb P} (V')$ for some vector subspace 
$V'$ of $V$. 
\end{Lemma}
\begin{proof}
The assumptions means that every tangent line of $Z$ is 
contained in $Z$. In particular, $Z$ contains all of its 
tangent planes. But any tangent plane of $Z$ is of the 
same dimension as $Z$ and consequently $Z$, and all
of its tangent planes, must coincide (this simple
argument was suggested by the referee).
\end{proof}

\subsection{Technical results}
\label{technical}

For technical reasons we will 
need the following setup : let $\mathcal Z$ denote a
projective variety and let 
$$f : Z \times Z \rightarrow \mathcal 
Z ,$$
denote a morphism. The projective morphism  
\begin{equation}
\label{mapZ}
\tau_f = (\tau_Z , f \circ \pi_Z ) : {\rm Bl}_\Delta(Z \times Z) 
\rightarrow {\rm Gr}_2(V)  \times  \mathcal Z  ,
\end{equation}
has a Stein factorization for which we use the notation
\begin{equation}
\label{Stein} {\rm Bl}_\Delta(Z \times Z) \xrightarrow{\mu_{f}} 
\mathcal B_{f} \rightarrow \mathcal {\rm Gr}_2(V) \times\mathcal  Z  .
\end{equation}
The restriction of $\tau_f$ to $E_Z$ is denoted by 
$$ \tau_{E,f} : E_Z \rightarrow {\rm Gr}_2(V)  \times  \mathcal Z.$$
More important is the map 
$$ \mu_{E,f} : E_{Z} \rightarrow { \mathcal S_{f}} := \mu_f\big(
E_Z\big),$$
induced by the restriction of $\mu_f$.
We claim

\begin{Lemma}
\label{semi-trivial}
The derived direct  images $R^i (\mu_{E,f})_* \mathcal O_{E_Z}$
are zero when $i>0$.
\end{Lemma}
\begin{proof}
As the second map $\mathcal B_{f} \rightarrow \mathcal 
{\rm Gr}_2(V) \times \mathcal Z$  of the Stein factorization 
(\ref{Stein}) is a finite map it suffices to
prove that $R^i (\tau_{E,f})_* \mathcal O_{E_Z}=0$
for $i>0$.
Consider an open affine subset $U$ of ${\rm Gr}_2(V)$
such that 
$$ \tau_E : E  \rightarrow {\rm Gr_2}(V),$$
is a trivial ${\mathbb P}^1$-bundle over $U$. 
Then we may consider $\tau_{E_Z}^{-1}(U)$ 
as a closed subvariety of ${\mathbb P}^1 \times U$. 
Embedding  $\tau_{E_Z}^{-1}(U)$ by the graph of 
$f \circ \pi_{Z}$  defines a closed 
embedding
$$ \iota : \tau_{E_Z}^{-1}(U) \hookrightarrow
Y := {\mathbb P}^1 \times U \times \mathcal Z.$$
The map 
$$ \tau_U :  \tau_{E_Z}^{-1}(U) = \tau_{E,f}^{-1}(U \times \mathcal Z) \rightarrow U \times \mathcal Z,$$
induced by the projection $p_{2,3}$ of $Y$ 
on the second and third coordinate,
coincides with the restriction of $\tau_{E,f}$ 
to the inverse image of $U \times \mathcal Z$. It thus suffices to 
prove that $R^i (\tau_U)_*  \mathcal O_{\tau_{E_Z}^{-1}(U)} = 0$ for $i>0$.  
Now apply the identity
$$ R^i (\tau_U)_*  \mathcal O_{\tau_{E_Z}^{-1}(U)}  =
R^i (p_{2,3})_* (\iota_* \mathcal O_{\tau_{E_Z}^{-1}(U)} ),$$
and the long exact sequence 
$$ \cdots \rightarrow R^1 (p_{2,3})_*  \mathcal I
\rightarrow  R^1 (p_{2,3})_*  \mathcal O_Y =0
\rightarrow R^1(p_{2,3})_* (\iota_* \mathcal O_{\tau_{E_Z}^{-1}(U)} )
\rightarrow 0 \rightarrow \cdots $$
associated to the trivial ${\mathbb P}^1$-bundle $p_{2,3}$,
and the short exacts sequence
$$ 0 \rightarrow \mathcal I \rightarrow \mathcal O_Y 
\rightarrow \iota_* \mathcal O_{\tau_{E_Z}^{-1}(U)} 
\rightarrow 0,$$
defining $\tau_{E_Z}^{-1}(U) $ as a closed subvariety in $Y$.  
\end{proof}

\begin{Lemma}
\label{birational}
Assume that $Z$ does not coincide with a closed
subvariety of ${\mathbb P}(V)$ of the form 
${\mathbb P}(V')$, for some vector subspace $V'$
of $V$. Then $\mu_{E,f}$ is birational
\end{Lemma}
\begin{proof}
Let $Y \subset {\rm Gr}_2(V) \times \mathcal Z$ denote 
the image of $\tau_f$.
We claim that there exists a point $y \in Y$ 
such that the fibre $\tau_f^{-1}(y)$ is nonempty and 
finite. To see this we use Lemma \ref{linear subspace} to 
obtain a point $b \in {\rm Gr}_2(V)$ such that the 
fibre $\tau_{E_Z}^{-1}(b)$ is nonempty and finite.
Assume, for a moment, that $\tau_Z^{-1}(b)$ is 
infinite : then $\pi_Z(\tau_Z^{-1}(b))$ is an
infinite closed subvariety of ${\mathbb P}(b)\times {\mathbb P}(b) = 
{\mathbb P}^1\times {\mathbb{P}^1}$ and thus ${\mathbb P}(b)$ is 
contained in $Z$. As a consequence 
$$
({\mathbb P}(b) \times {\mathbb P}(b)) \setminus 
\Delta({\mathbb P}(b)) \times \{ b \},
$$
is a subset of ${\rm Bl}_\Delta(Z  \times Z)$
and thus, by taking the closure, we find that
$$ {\mathbb P}(b) \times \{ b \} \subset E_Z.$$
But then ${\mathbb P}(b) \times \{ b \}$ is a subset of the 
finite set $\tau_{E_Z}^{-1}(b)$, which is a 
contradiction. It follows that $\tau_Z^{-1}(b)$
is finite and nonempty. Choose an element 
$y$ in  $\tau_f( \tau_Z^{-1}(b))$.
As a subset of $\tau_Z^{-1}(b)$  the 
set $\tau_f^{-1}(y)$ is then finite.

Let now $Y_0$ denote the nonempty set of points in 
$Y$ where the associated fibre of $\tau_f$ is finite.
Then $Y_0$ is an open subset of $Y$
(\cite[Cor. I.8.3]{Mumford}). 
 It follows 
that $\mu_f$ induces an isomorphism between
$\mathcal B_{f}$ and ${\rm Bl}_\Delta(Z \times Z)$
over $Y_0$
$$ \mu_{f,0} : {\tau_f}^{-1}(Y_0) \xrightarrow{\simeq} \mu_{f}
({\tau_f}^{-1}(Y_0)).$$  
It thus suffices to prove that the intersection 
of $E_Z$ and ${\tau_f}^{-1}(Y_0)$ is nonempty.
But this is clear as $\tau_{E_Z}^{-1}(b)$ is 
a nonempty subset of  ${\tau_f}^{-1}(Y_0)$.
\end{proof}

From now on we will assume that $f : Z \times Z
\rightarrow \mathcal Z$ is the product 
$(f_1,f_2)$ of two morphisms
$$ f_i : Z \rightarrow \mathcal Z_i, ~i=1,2. $$
We can then prove.
 
\begin{Lemma}
\label{conn-fibres}
The 
fibres of $\mu_{E,f}$ are connected.
 \end{Lemma}
\begin{proof}
Let $z$ denote an element in $\mathcal S_{f}$
and let $(b,x)$ denote the image of  $z$ in 
under the second morphism 
\begin{equation}
\label{Stein2}
 \mathcal B_{f} \rightarrow {\rm Gr}_2(V) \times \mathcal Z
\end{equation}
of the Stein factorization (\ref{Stein}). 
As $\mu_{E,f}^{-1}(z)  \subset \mu_f^{-1}(z)$
and $\mu_f^{-1}(z)$ is connected we may assume that 
$\mu_f^{-1}(z)$ is infinite. Consequently the intersection
$Z \cap {\mathbb P}(b)$ must also be infinite and thus 
equal to $ {\mathbb P}(b)$. It follows that  
$$ {\mathbb P}(b) \times  {\mathbb P}(b) \times
\{ b \} \subset {\rm Bl}_\Delta(Z \times Z).$$
This leads to  the inclusion
\begin{equation}
\label{inclusion}
\mu_f^{-1}(z) \subset \tau_f^{-1}(b,x) =  (f_1^{-1}(x_1) \cap {\mathbb P}(b)) 
\times (f_2^{-1}(x_2) \cap {\mathbb P}(b)) \times \{ b \},
\end{equation}
where we have used the notation $x=(x_1,x_2) \in \mathcal Z$, 
with $x_i \in \mathcal Z_i$ for $i=1,2$. 
As $\mu_f$ and $\tau_f$ only differ by a finite morphism
it follows 
 that $\tau_f^{-1}(b,x) $
is a disjoint union of $\mu_f^{-1}(z) $ with another 
closed (possibly empty) subset of $\tau_f^{-1}(b,x) $.
At the same time $\mu_f^{-1}(z) $ is connected 
and thus (\ref{inclusion}) implies that $\mu_f^{-1}(z) $ 
is of one of the forms
$$ {\mathbb P}(b)  \times {\mathbb P}(b) \times  \{ b \}, ~
\{ l \}   \times {\mathbb P}(b) \times 
\{ b \}, ~ 
{\mathbb P}(b)  \times \{ l \} \times 
\{ b \},$$
for some line $l$ contained in $b$. We conclude 
that $\mu_{E,f}^{-1}(z)$ is  either equal to
$$ {\mathbb P}(b) \times \{ b \} \subset E_Z,$$
or of the form
$$  \{ l \}  \times \{ b \} \subset E_Z.$$
In both cases $\mu_{E,f}^{-1}(z)$ is connected.

\end{proof}

\subsection{Proof of Theorem \ref{conjecture}}

We continue the notation of Section \ref{technical}.  The proof of
Theorem \ref{conjecture} is built from the following two results.

\begin{Proposition}
\label{prop1}
Assume that $E_Z$ admits a Frobenius splitting.
Let $\mathcal L$ (resp. $\mathcal M$) 
denote a very ample (resp. globally generated) 
line bundle on $Z$ and let $j > 0$ denote an 
integer. Then 
\begin{equation}
\label{tangent-van}
{\rm H}^i\big( Z , S^j \Omega^1_{Z} \otimes {\mathcal L}^{2j}
\otimes \mathcal M \big) = 0 ~, \text { for } i>0.
\end{equation}
\end{Proposition}
\begin{proof}
 
We assume that the embedding $Z \subset {\mathbb P}^N$
is defined by the very ample line bundle $\mathcal L$, and 
that the map $f$, of Section \ref{technical}, is the composition 
$$f : Z \times Z \rightarrow Z \rightarrow \mathcal Z := 
{\mathbb P}\big({\rm H}^0(\mathcal M)^\vee \big),$$
where the first map is projection on the first coordinate 
while the second map is the projective morphism defined
by the globally generated line bundle $\mathcal M$.
Let $\mathcal O_{\mathcal M}(1)$
denote the ample generator of the Picard group of 
${\mathbb P}\big({\rm H}^0(\mathcal M)^\vee\big)$.
By (\ref{iso-line}) the pull-back of $\mathcal O_{2,V}(j)
\boxtimes \mathcal O_{\mathcal M}(1)$ by 
$$ \tau_{E,f} : E_Z \rightarrow {\rm Gr}_2(V) \times \mathcal Z,$$
is then the line bundle
$$ \mathcal L_j = \mathcal O(-j E_Z)_{| E_Z} \otimes \pi_{E_Z}^* \big( \mathcal L^{2j} \otimes \mathcal M\big) ,$$ 
on $E_Z$. Consider the Stein factorization 
$$ E_Z \xrightarrow{\tilde \mu_{E,f}} \tilde {\mathcal S_{f}} \rightarrow \mathcal S_{f},$$
of $\mu_{E,f}$. By Lemma 
\ref{semi-trivial} and the definition of the Stein factorization,
the map $ \tilde \mu_{E,f}$ is a rational morphism, i.e. 
$$
R^i (\tilde \mu_{E,f})_* \mathcal O_{E_Z} =
\begin{cases}
\mathcal O_{\tilde {\mathcal S_{f}}} & \text{ if } i=0, \\
0 & \text{ if } i>0.
\end{cases}
$$
Moreover, the pull back $\tilde{ \mathcal L}_j$ of 
$\mathcal O_{2,V}(j) \boxtimes \mathcal O_{\mathcal M}(1)$ 
 by the finite morphism
\begin{equation}
\label{finite map}
\tilde {\mathcal S_{f}} \rightarrow \mathcal S_{f}
\rightarrow \mathcal B_{f} \rightarrow {\rm Gr}_2(V) 
\times \mathcal Z,
\end{equation}
is an ample line bundle  on $\tilde{\mathcal  S_{f}}$ whose pull back by 
$\tilde \mu_{E,f}$ coincides with $\mathcal L_j$. 
As  $\tilde {\mathcal S_{f}}$ is Frobenius split (by push-down
of the Frobenius splitting on $E_Z$ \cite[Lemma 1.1.8]{BrionKumar2005} ) it follows that 
the higher cohomology of  $\tilde{ \mathcal L}_j$,
and hence of  ${ \mathcal L}_j$, is trivial \cite[Thm.1.2.8]{BrionKumar2005}.
Notice finally that by \cite[Ex. III.8.4]{Hartshorne} the cohomology 
of ${\mathcal L}_j$ and the direct image
$$ (\pi_{E_Z})_* {\mathcal L}_j =
 S^j \Omega^1_{Z} \otimes {\mathcal L}^{2j}
\otimes \mathcal M,$$
coincide. Here we use that the identification 
$ (\pi_{E_Z})_* \mathcal O_{E_Z}(-j E_Z) = S^j \Omega_Z^1$.
This ends the proof.  
\end{proof}

\begin{Proposition}
\label{prop2} 
Assume that the blow-up ${\rm Bl}_{\Delta}(Z \times Z)$ admits 
a Frobenius splitting which is compatible  with $E_Z$. 
Let $\mathcal L$ denote a very ample line bundle on $Z$ 
and let $\mathcal M_1$ and $\mathcal M_2$ denote globally 
generated line bundles on $Z$. Let $j>0$ denote an integer. 
Then
\begin{equation}
\label{Bl-van2}
{\rm H}^i\big(Z \times Z, \mathcal I_\Delta^{j+1} \otimes
((\cL^j\otimes\cM_1) \boxtimes (\cL^j \otimes 
\mathcal M_2 ))\big)  = 0,  \text{ for } i >0,
\end{equation}
where $\mathcal I_\Delta$ denotes the sheaf
of ideals defining the diagonal in $Z \times Z$.
\end{Proposition}
\begin{proof}
We will assume that $\mathcal L$ is the line bundle
defining the embedding $Z \subset {\mathbb P}^N$,
and that 
$$ f_i : Z \rightarrow \mathcal Z_i := {\mathbb P}^0
\big( {\rm H}^0(\mathcal M_i)^\vee \big), ~i=1,2,$$
are the maps defined by the globally generated line
bundles $\mathcal M_1$ and $\mathcal M_2$ . 
Let $\mathcal L_j$ denote the line bundle 
$$ \mathcal L_j = \mathcal O(-j E_Z) \otimes \pi_Z^* 
\big( (\mathcal L^{j} \otimes \mathcal M_1)
\boxtimes (\mathcal L^{j} \otimes \mathcal M_2 )
\big),$$
on ${\rm Bl}_\Delta(Z \times Z)$. 
We claim that the restriction morphism
\begin{equation}
\label{gaussian}
 {\rm H}^0 \big( {\rm Bl}_\Delta(Z \times Z), \mathcal L_j
\big) \rightarrow {\rm H}^0 \big( E_Z, 
\mathcal L_j \big),
\end{equation}
is surjective. To see this 
let $\mathcal O_i(1)$, for $i=1,2$, 
denote the ample generator of the Picard group
of $\mathcal Z_i$. Consider the ample line bundle
$$\tilde{M}_j = \mathcal O_{2,V}(j) \boxtimes 
\mathcal O_1(1) \boxtimes \mathcal O_2(1),$$
on ${\rm Gr}_2(V) \times \mathcal Z$
and let $\tilde{\mathcal L}_j$ denote the ample
pull back of $\tilde{\mathcal M}_j$ to 
$\mathcal B_{f}$ by the finite morphism in 
(\ref{Stein}). Then by  (\ref{iso-line})
the line bundle  $\mathcal L_j$ is the pull 
back of $\tilde{\mathcal L}_j$ by $\mu_f$.
In particular, as $\mu_f$ is part of a Stein
factorization we obtain an identification 
$${\rm H}^0 \big( {\rm Bl}_\Delta(Z \times Z), \mathcal L_j
\big) = {\rm H}^0 \big( \mathcal B_{f}, \tilde{\mathcal L}_j
\big) .$$
Assume, for a moment, that $Z$ is not of the 
form ${\mathbb P}(V')$ as in the assumptions
of Lemma \ref{birational}.  Then $\mu_{E,f}$ 
is a birational morphism with connected fibres
by Lemma \ref{birational} and Lemma \ref{conn-fibres}.
Moreover, by push-forward   of the Frobenius splitting 
on ${\rm Bl}_{\Delta}(Z \times Z)$ we know that 
$ \mathcal B_{f}$ is Frobenius split compatibly 
with $\mathcal S_{f}$  \cite[Lemma 1.1.8]{BrionKumar2005}. 
Thus by \cite[Ex. 1.2.E(3)]{BrionKumar2005} the 
variety $\mathcal S_{f}$ is normal, and hence 
\begin{equation}
\label{identity}
  {\rm H}^0 \big( E_Z , \mathcal L_j
\big) = {\rm H}^0 \big(\mathcal  S_{f}, \tilde{\mathcal L}_j
\big) ,
\end{equation}
by Zariski's main theorem.
Thus to prove (\ref{gaussian}) it suffices to prove that the restriction
map
$$ {\rm H}^0 \big( \mathcal B_{f}, \tilde{\mathcal L}_j
\big) \rightarrow
{\rm H}^0 \big( \mathcal S_{f}, \tilde{\mathcal L}_j
\big), 
$$
is surjective. As $\mathcal S_{f}$ is compatibly Frobenius 
split in $\mathcal B_{f}$ and as $\tilde{\mathcal L}_j$
is ample the latter follows by general theory of Frobenius 
splitting \cite[Thm.1.2.8]{BrionKumar2005} . Consider next the case $Z={\mathbb P}(V')$. If 
either $\mathcal M_1$ or $\mathcal M_2$ are ample 
then $\mu_{E,f}$ is easily seen to be an isomorphism
and we may argue as above. This leaves us with the 
case $\mathcal M_1 = \mathcal M_2 = \mathcal O_Z$.
Then $\mathcal Z$ is just a 1-point space and thus 
$\mathcal S_{f} = {\rm Gr}_2(V')$ while $\mu_{E,f}$
coincides with $\tau_{E_Z}$ which is a $\mathbb P^1$-bundle
over ${\rm Gr}_2(V')$. So again we obtain the identification
(\ref{identity}). This proves the claim about the surjectivity 
of (\ref{gaussian}). 

As the blow-up map satisfies
$$ R^i(\pi_{Z})_* \mathcal O(-j E_Z)
= 
\begin{cases}
(\mathcal I_\Delta)^j & \text{ if } i=0,\\
0 & \text{ if } i>0,
\end{cases}
$$
we may reformulate the statement (\ref{Bl-van2})
as 
$$ {\rm H}^i\big({\rm Bl}_\Delta( Z \times Z), 
\mathcal L_j \otimes \mathcal O(-E_Z) \big)  = 0,
\text{ for } i>0.$$
To prove the latter  we consider the short exact sequence
$$ 0 \rightarrow \mathcal O(-E_Z) \rightarrow 
\mathcal O_{{\rm Bl}_\Delta(Z \times Z)} 
\rightarrow \mathcal O_{E_Z} \rightarrow 0,$$
and apply Proposition \ref{prop1}. It 
follows that it suffices to prove
$$ {\rm H}^i\big({\rm Bl}_\Delta( Z \times Z), 
\mathcal L_j \big)  = 0,
\text{ for } i>0.$$
As $E_Z$
is compatibly Frobenius split divisor  in ${\rm Bl}_\Delta(Z \times Z)$
we have by \cite[Lemma 1.4.11]{BrionKumar2005}  an inclusion (of abelian groups)
$$ {\rm H}^i \big({\rm Bl}_\Delta(Z \times Z),
\mathcal L_j \big) \subset  {\rm H}^i 
\big({\rm Bl}_\Delta(Z \times Z), \mathcal L_j^p 
\otimes \mathcal O((p-1) E_Z) \big).$$
Thus, as ${\rm Bl}_\Delta(Z \times Z)$ is 
Frobenius split, it suffices to show
that the line bundle 
$$  \mathcal L_j^p \otimes \mathcal O((p-1) E_Z)
= \mathcal O((p(1-j)-1) E_Z) \otimes \pi_Z^* 
\big( (\mathcal L^{pj} \otimes \mathcal M_1^p)
\boxtimes (\mathcal L^{pj} \otimes \mathcal M_2^p )
\big),
$$
is ample on ${\rm Bl}_\Delta(Z \times Z)$. But 
the latter line bundle is by  (\ref{iso-line})
isomorphic to the restriction to ${\rm Bl}_\Delta(Z \times Z)$
of the line
bundle
\begin{equation}
\label{ample}
\big(\mathcal M_1^p \otimes \mathcal L^{(p-1)} \big)
\boxtimes \big(\mathcal M_2^p \otimes \mathcal L^{(p-1)} \big)
\boxtimes \mathcal O_{2,V}\big(p(j-1)+1 \big),
\end{equation}
on  $Z \times Z \times {\rm Gr}_2(V)$. Here $\mathcal O_{2,V}\big(1\big)$ denotes the ample 
generator of the Picard group of ${\rm Gr}_2(V)$.
As the line bundle (\ref{ample}) is 
ample this ends  the proof.
\end{proof}

Theorem \ref{conjecture} is now a direct consequence of 
Proposition \ref{prop2}.

\bibliographystyle{amsplain}

\providecommand{\bysame}{\leavevmode\hbox to3em{\hrulefill}\thinspace}
\providecommand{\MR}{\relax\ifhmode\unskip\space\fi MR }
\providecommand{\MRhref}[2]{%
  \href{http://www.ams.org/mathscinet-getitem?mr=#1}{#2}
}
\providecommand{\href}[2]{#2}

\end{document}